\documentclass[envcountsect,envcountsame]{svjour3}

\usepackage[margin=2.5cm, lmargin=2cm, rmargin=2cm, bmargin=2cm]{geometry}

\makeatletter
\renewcommand\labelitemi{$\m@th\bullet$}
\renewcommand\labelitemii{$\m@th\bullet$} 
\makeatother

\usepackage[utf8]{inputenc} 	
\usepackage[T1]{fontenc}    	
\usepackage{url}            	
\usepackage{amsfonts}       	
\usepackage{microtype}      	
\usepackage[dvipsnames]{xcolor} 
\usepackage{amsmath}        	
\usepackage{float}          	
\usepackage{mathtools}      	
\usepackage{bm}              	
\usepackage{parskip}        	
\usepackage{orcidlink}      	
\usepackage{lmodern}            
\usepackage{framed}

\definecolor{lessred}{HTML}{FF5010} 
\definecolor{lessblue}{HTML}{4BAAC8} 
\hypersetup{
    colorlinks=true,
    linkcolor=lessred,
    filecolor=magenta,      
    urlcolor=magenta,
    citecolor=lessblue
    }

\newcommand{\R}{\mathbb R}
\newcommand{\E}{\mathbb E}
\newcommand{\N}{\mathbb N}
\renewcommand{\H}{\mathcal H}
\newcommand{\Exp}[1]{\mathbb{E}\mathopen{}\left[{#1}\right]\mathclose{}}
\newcommand{\Expt}[1]{\mathbb{E}_t\mathopen{}\left[{#1}\right]\mathclose{}}
\newcommand{\ExpD}[2]{\mathbb{E}_{{#1}}\mathopen{}\left[{#2}\right]\mathclose{}}
\newcommand{\argmin}{\operatorname{argmin}}
\newcommand{\dist}{\operatorname{dist}}
\newcommand{\prox}{\operatorname{prox}}
\newcommand{\proj}{\operatorname{proj}}

\renewcommand{\Leftrightarrow}{\iff}
\newcommand{\rate}{\rho}
\newcommand{\opt}{{\operatorname{opt}}}
\newcommand{\crit}{\operatorname{crit}}

\newcommand{\citep}{\cite}

\usepackage[normalem]{ulem}
\usepackage{amsopn}

\smartqed
\let\svjourendproof\endproof
\renewcommand{\endproof}{\qed\svjourendproof}

\title{Bias-Optimal Bounds for Stochastic Gradient Descent}

\author{
  Daniel Cortild$^{1,2}$ \orcidlink{0000-0002-3278-1716}, 
  Lucas Ketels$^{2,3}$ \orcidlink{0009-0009-0669-8757}, 
  Juan Peypouquet$^2$ \orcidlink{0000-0002-8551-0522}, 
  Guillaume Garrigos$^3$ \orcidlink{0000-0002-8613-5664}
}

\authorrunning{Daniel Cortild, Lucas Ketels, Juan Peypouquet, Guillaume Garrigos}

\institute{
    $^1$University of Oxford, Mathematical Institute, Woodstock Road, OX2 6GG Oxford, United Kingdom \\
    $^2$University of Groningen, Bernoulli Institute, Bernoulliborg, 9747AG Groningen, Netherlands \\
    $^3$Université Paris Cité and Sorbonne Université, CNRS, Laboratoire de Probabilités, Statistique et Modélisation, F-75013 Paris, France \\
    \texttt{daniel.cortild@maths.ox.ac.uk, l.ketels@rug.nl, j.g.peypouquet@rug.nl, garrigos@lpsm.paris}
}

\date{
Submitted: November 13, 2026 
}

\begin{document}

\maketitle

\vspace{2em}

\begin{abstract}
    The non-asymptotic analysis of Stochastic Gradient Descent (SGD) typically yields bounds that decompose into a bias term and a variance term. In this work, we focus on the bias component and study the extent to which SGD can match the optimal convergence behavior of deterministic gradient descent. 
    Assuming only (strong) convexity and smoothness of the objective, we derive new bounds for the full range of constant step-sizes $\gamma L \in (0,2)$.
    This includes critical and large step-size regimes that were previously unexplored without additional variance assumptions, in both the convex and strongly convex cases. 
    Those bounds are bias-optimal, in the sense that the bias term coincides with the worst-case rate of gradient descent, for all non-critical step-sizes. 
    For critical step-sizes, we show that no bounds can be obtained with an optimal bias term, but that this optimal bias can be approximated arbitrarily closely.
    As an immediate application, we provide new complexity bounds for the Stochastic Proximal algorithm, without the need of regularity assumptions. 
\end{abstract}


\section{Introduction}

Given a family $\{f_i\}_{i\in \mathcal I}$ of convex and smooth real-valued functions 
$f_i\colon \H\to \R$, we consider the problem given by
\begin{equation}\tag{$\mathbb E$-minimization}\label{LR2::eq:finite}
    \min_{x\in \H}f(x),\quad \text{where }f\coloneqq \ExpD{i\sim \mathcal D}{f_i},
\end{equation}
where $\mathcal D$ is a probability distribution over $\mathcal I$.
We assume $\argmin(f)$ to be nonempty. When the indexing set $\mathcal I$ is finite, note that we recover the typical finite sum minimization problem.

The Stochastic Gradient Descent (SGD) algorithm \citep{robbins_stochastic_1951} is one of the most popular approaches for solving Problem \eqref{LR2::eq:finite}, and is widely used for large-scale machine learning and stochastic optimization problems. In its standard version, at each iteration, SGD performs a gradient step using one of the functions, which is chosen at random. More precisely, given the iterate $x_t$, it picks $i_t$ independently and identically distributed according to $\mathcal D$, and computes $x_{t+1}$ via
\begin{equation}\label{S1::eq:SGD}\tag{SGD}
    x_{t+1}=x_t-\gamma \nabla f_{i_t}(x_t),
\end{equation}
where $\gamma>0$ is the step-size, and $x_0 \in \mathcal{H}$ is deterministically predetermined. Intuitively, SGD is able to progress towards a solution of the problem because it follows a direction that, on average, is an unbiased estimator for $\nabla f$. The simplicity and efficiency of SGD make it the go-to algorithm for training deep neural networks and other models on massive datasets, where computing the entire gradient $\nabla f$ is intractable or too expensive. 
While we keep our focus on vanilla SGD, all the results in this paper also apply to mini-batching. 
This is a simple consequence of the arguments laid out in \cite{gower_sgd_2019}, which we also provide in Appendix \ref{SA::sec:extensions} for completeness and for the reader's convenience.

\textbf{Complexity of SGD.} Quantitative results on the behavior of SGD consist in upper bounds for a metric $\Delta(T)$, which is to be interpreted as an optimality gap at iteration $T$. 
Typical choices for such a metric include the expected solution gap
$\Exp{\Vert x_T - x_*\Vert^2}$,  
the expected function value gap $\Exp{f(x_T) - \inf f}$, 
the expected average function value gap $\tfrac{1}{T}\sum_{t=0}^{T-1} \Exp{f(x_t) - \inf f}$, or the expected ergodic function value gap $\Exp{f(\bar x_T) - \inf f}$, where  $\bar x_T$ is the Ces\`aro average of the iterates.
In most known results \cite{garrigos_handbook_2024}, these upper bounds are expressed as the sum of two terms, taking the form 
\begin{equation*}
\Delta(T)\le {\rm Bias}(T)+{\rm Variance}(T),
\end{equation*}
where, typically, the {\it bias term} vanishes when $T \to + \infty$ and the \emph{variance term}  can be made arbitrarily small by taking a small enough step-size.
For strongly convex and smooth problems,  the metric $\Delta(T)$ will typically be the expected solution gap.
For convex smooth problems, most proofs construct a bound on the expected average function value gap, and then derive a bound on the expected ergodic function value gap through Jensen's inequality.

This paper discusses and addresses the following question: \emph{what are the best non-asymptotic bounds for SGD?}
In recent years, remarkable efforts have been made to answer this question for various deterministic optimization methods.
These results include the optimal convergence rates of gradient descent \citep{drori_performance_2014,taylor_smooth_2017,rotaru_exact_2024}, proximal gradient descent \citep{taylor_exact_2018}, gradient descent with line-search \citep{de_klerk_worst-case_2017}, to name a few.
In contrast, 
obtaining an optimal non-asymptotic complexity of stochastic algorithms is significantly more challenging and the optimal complexity of SGD remains nowadays unknown. 
In this work, we establish new and improved bounds for SGD, that are provably optimal \emph{in some sense}. These bounds take the form
\begin{equation*}
\Delta(T)\le {\rm Bias}(T)\cdot \|x_0-x_*\|^2+{\rm Variance}(T)\cdot \sigma_*^2,
\end{equation*}
where $x_*$ is a solution of the problem and $\sigma^2_* = \mathbb{V}[\nabla f_i(x_*)]$ is the \textit{solution gradient variance} constant (see more details in Section \ref{SA::sec:assumptions}).
We note that in the deterministic setting, where SGD reduces to GD and where $\sigma^2_*=0$, such bound reduces to $\Delta(T)\le \text{Bias}(T)\cdot \|x_0-x_*\|^2$. 
We will therefore say that a bound is \textit{bias-optimal} if the 
bias term matches the optimal non-asymptotic convergence rate of GD.
In particular, a bias-optimal bound would guarantee that SGD attains an unimprovable convergence rate of the metric $\Delta (T)$ in the \textit{interpolation regime} (that is when $\sigma^2_* = 0$). We note, however, that it is not clear whether general bounds for SGD always take the form of a gradient descent bound plus an additional variance term. In this work, we solely focus on bounds of the given form.

\textbf{Contributions.}
In the convex case, we establish new and improved bounds in expectation for the average function value gap, which automatically convert into bounds for the more standard ergodic function value gap. Our bounds, obtained in Theorems \ref{S2::thm:main_small}, \ref{S2::thm:main_1} and \ref{S2::thm:main_large}, with simplified notation for ease of presentation, are
\begin{equation*}
    {\rm Bias}(T) \simeq
    \begin{cases}
        \tfrac{1}{2 \gamma T}  & \text{ if } \gamma L \in (0,1), \\
        \tfrac{1}{2(1 - \varepsilon) \gamma T} & \text{ if } \gamma L =1, \\
        \tfrac{1}{2 \gamma (2 - \gamma L) T} & \text{ if } \gamma L \in (1,2),
    \end{cases}
    \qquad
    {\rm Variance}(T) \simeq
    \begin{cases}
        \tfrac{\gamma}{2 (1-\gamma L)}  & \text{ if } \gamma L \in (0,1), \\
        \tfrac{\gamma(1+\varepsilon)}{2\varepsilon (1 - \varepsilon)} & \text{ if } \gamma L =1, \\
        \tfrac{\exp(T)}{2 - \gamma L} & \text{ if } \gamma L \in (1,2),
    \end{cases}    
\end{equation*}
where $\varepsilon\in (0, 1)$ is arbitrary for the critical step-size.

For short step-sizes $\gamma L \in (0,1)$, our bounds provide better constants than what was previously known.
More interestingly, we prove that our bound is bias-optimal for both the expected average function value gap and the expected ergodic function value gap.
The setting $\gamma L \in [1,2)$ is more delicate to summarize. 
Our bounds are entirely new, as this setting had not been studied prior to the initial public release of this work. More importantly, they reveal and suggest non-trivial interactions between the bias and variance terms which, to the best of our knowledge, were never observed before.
We now split our discussion between the critical and large step-sizes.

For large step-sizes $\gamma L \in (1,2)$, we prove that our bound is bias-optimal for the expected average function value gap.
A striking feature of this bound is that the variance term has an impractical exponential growth with time.
This exponential term disappears if we allow for a small deviation for the optimal bias term (see Theorem \ref{S2::th:large subopt} for more details).
Can one obtain a bias-optimal bound with a \emph{bounded} variance for large step-sizes? 
The question remains open. 
Of course, those considerations about the variance term are irrelevant in the interpolation regime, where $\sigma_*^2=0$.

At the critical step-size $\gamma L = 1$, the optimal deterministic bias coefficient is $\tfrac{L}{2T}$. By considering the limit of the optimal bias as $\gamma L\to 1$, we approach this optimal bias coefficient, but the variance coefficient in our bound diverges. 
We show first that it is impossible to derive a bound for SGD with $\gamma L = 1$ whose bias term is $\tfrac{L}{2T}$. 
Second, we provide bounds whose bias term are arbitrarily close to $\tfrac{L}{2T}$, which in this scenario is best one can do.

We highlight that the study of the critical step-size $\gamma L=1$ in the convex smooth setting is novel. 
This in turn unlocks the study of the \emph{stochastic proximal algorithm}, which can be viewed as an instance of SGD applied to regularized functions if the step-size is the critical one. 
To the best of our knowledge, we establish the first complexity guarantees for this method in a general convex nonsmooth setting, considering possibly functions which do not have a full domain.
This is presented in Section \ref{SA::sec:prox}.

Let us now turn to the smooth strongly convex case, where each function $f_i$ is $L$-smooth and $\mu$-strongly convex.
Our performance metric is the expected distance to the solution, namely $\Delta(T) = \Exp{\| x_T-x_*\|^2}$.
Conceptually, our results are of the same nature as in the convex setting.
We obtain bounds for the full range of step-sizes $\gamma L \in (0,2)$, while the existing literature \citep{bach_non-asymptotic_2011,needell_stochastic_2016,gower_sgd_2019} considered at best $\gamma L < 1$.
Extending the range of valid step-sizes allows us to consider the critical step-size $\gamma= \tfrac{2}{\mu + L} \geq 1/L$, for $\mu \leq L$, which is known to be optimal in the deterministic setting.
For every non-critical step-size we obtain 
${\rm Bias}(T) = \phi^{2T}$, where 
\begin{equation*}
    \phi = \max\{ 1 - \gamma \mu ; \gamma L -1 \}
\end{equation*}
governs the optimal rate of GD.
The constant $\phi$ is best when the step-size is critical, and equal to $\tfrac{L-\mu}{L+\mu}$. 
Similarly to the convex case, we show again that it is impossible to derive a bound for SGD whose bias term is governed by $\tfrac{L-\mu}{L+\mu}$, but that it can be made arbitrarily close.
\textbf{Proof Strategy.} For each fixed time horizon $T$ and $t=0, \dots, T$, we define a time-dependent random energy sequence 
\begin{equation*}
    E_t=a_t\cdot \|x_t-x_*\|^2+\rate\cdot \sum_{s=0}^{t-1}\big(f(x_s)-\min f\big) - \sum_{s=0}^{t-1}e_s\sigma_*^2,
\end{equation*}
where $\rate$, $(a_t)$ and $(e_t)$ are nonnegative parameters. 
The first term in the Lyapunov energy is the distance to the solution $\Vert x_t - x_* \Vert^2$, a classical term which decreases for deterministic gradient dynamics.
The second term involves the function gap $f(x_t) - \min f$, which also typically decreases for gradient descent.
We choose to replace the standard term $t(f(x_t) - \min f)$ with the sum of the past function gaps, which is of the same order in time.
To the best of our knowledge, this choice is new in the context of stochastic optimization, and we believe this choice to be critical to be able to derive bias-optimal bounds.
The third term is a \emph{negative} cumulated sum, 
introduced to compensate the fluctuations caused by the variance in the SGD algorithm, and hence to allow the Lyapunov energy to decrease.

A decrease of $\Exp{E_t}$ along the iterations implies that $\Exp{E_T} \leq E_0$, which translates into a bound for SGD (see Section~\ref{sec:bounds-decrease} for more details).
Our strategy is therefore to identify \textit{admissible} Lyapunov parameters, which are those that make the energy decrease in expectation.
Among those admissible parameters, our goal is to find the ones minimizing, whenever possible, the bias term in our bound.
To this end, we relied on the \textit{Performance Estimation Problem} (PEP) methodology, which originally appeared in \cite{drori_performance_2014}, and was further developed in \cite{taylor_smooth_2017}, and its use for systematic design of Lyapunov functions \cite{taylor_stochastic_2019,lessard_analysis_2016,fercoq_defining_2024,upadhyaya_automated_2025}.
These numerical explorations helped us to navigate the space of admissible Lyapunov parameters and to obtain initial guesses for some parameter sequences, which eventually led us to our theoretical proofs. 

\textbf{Related Work.} 
In early studies of SGD in the convex (smooth) framework, it was standard to make strong assumptions on the stochastic gradients $\nabla f_{i_t}(x_t)$.
The oldest, and possibly the best-known, are the uniform boundedness of the variance, and its variant, the uniform boundedness of the gradients, namely \citep{robbins_stochastic_1951,nemirovski_robust_2009,rakhlin_making_2012,schmidt_fast_2013}
$$\sup_{x\in \mathcal H}\Exp{\|\nabla f_i(x) - \nabla f(x)\|^2}<+\infty\quad \text{or}\quad \sup_{x\in \mathcal H}\Exp{\|\nabla f_i(x) \|^2}<+\infty.$$
Such assumptions date back to the early works of the 1950s \citep{chung_stochastic_1954,sacks_asymptotic_1958} and continued to be employed until recently \citep{polyak_acceleration_1992a,nemirovski_robust_2009}.
A culminating contribution to the nonasymptotic analysis of SGD under bounded variance assumptions is \cite{taylor_stochastic_2019}, where new and sharp convergence rates were established.
Over time, these variance assumptions have been progressively relaxed and replaced by conditions that control the gradient or its variance through quantities such as $\|x\|$ or $\|\nabla f(x)\|$, we refer to the overview in \cite{alacaoglu_towards_2025} and the references therein.
Nevertheless, variance-based assumptions can be unrealistic in practice \citep{bottou_optimization_2018,nguyen_sgd_2018}.
Moreover, uniform boundedness of the gradients can never be satisfied for strongly convex functions.

A new line of research was initiated in \cite{bach_non-asymptotic_2011} by removing any such variance assumptions, followed by the more recent works 
\citep{needell_stochastic_2016,nguyen_sgd_2018,gower_sgd_2019,gower_sgd_2021,khaled_better_2023,demidovich_guide_2023}. 
The core idea is to exploit the convexity and smoothness to obtain a variance transfer inequality, bounding the variance at any point with a term involving $\sigma^2_*$. The new gold standard for the analysis of SGD quickly became to not make any variance assumption when assuming convexity and smoothness, a framework into which our work fits.

We point out that convexity and smoothness are sufficient but not necessary conditions to obtain complexity results: all one needs is for some variance transfer inequality to hold.
We mention for instance the expected smoothness assumption \citep{gower_stochastic_2021,gower_sgd_2019,khaled_better_2023} and the ABC condition \citep{demidovich_guide_2023}, which were shown to be sufficient to establish performance guarantees. These assumptions are strictly weaker than convexity and smoothness, enabling the analysis of SGD in a nonconvex framework. While our analysis could also be carried out under weaker assumptions than convexity and smoothness, it is unclear how practical such assumptions would be, and we do not delve further into this in this work.

\textbf{Note on Concurrent Work.} 
In the convex setting, our work focuses on rates on the \textit{ergodic} function value gap $f(\bar x_T)-\min f$, which is currently the standard for SGD without uniformly bounded variance. In contrast, when making a stronger assumption on the stochastic gradients, it is standard to obtain rates on the \textit{last-iterate} function value gap \cite{bach_non-asymptotic_2011,taylor_stochastic_2019,liu_revisiting_2023}. At the time of writing this document, it was unclear whether SGD without stronger variance assumptions could enjoy last-iterate rates, hence justifying our choice of performance metric. After the initial publication of the manuscript, two works have answered this question positively \cite{attia_fast_2025,garrigos_lastiterate_2025}, opening the door for the possibility to extend our results to the setting of last-iterate rates.

\textbf{Paper Organization.}
This paper is organized as follows: We present the formal problem and the assumptions in Section \ref{SA::sec:assumptions}. In Section \ref{S3::sec__}, we outline the proof format of our main results. Sections \ref{S2::sec} and \ref{S3::sec:str}  are devoted to our results (new bounds, most of which being bias-optimal) in the convex and strongly convex settings, respectively. In Section \ref{SA::sec:prox} we present the application of our results to the Stochastic Proximal algorithm. 
The technical proofs are gathered in the appendices.

\section{Preliminaries}\label{SA::sec:assumptions}

We first formalize the problem statement presented in Problem \eqref{LR2::eq:finite}.

\begin{problem}[Main problem]\label{Pb:main}
    Let $\mathcal H$ be a separable real Hilbert space with associated inner product $\langle \cdot, \cdot \rangle$ and induced norm $\|\cdot\|$.
    Let $\{f_i\}_{i\in \mathcal I}$ be a family of real-valued functions $f_i \colon \mathcal H \to \mathbb{R}$, where $\mathcal I$ is a (possibly infinite) set of indices.
    We consider the problem of minimizing $f\coloneqq \Exp{f_i}$, where the expectation is taken over the indices $i \in \mathcal I$, with respect to some probability distribution $\mathcal{D}$ over $\mathcal I$. We assume that
    \begin{enumerate}
        \item the problem is well-defined, in the sense that $i \mapsto f_i(x)$ is $\mathcal{D}$-measurable, and that $\Exp{f_i(x)}$ is finite for every $x \in \mathcal{H}$;
        \item the problem is well-posed, in the sense that ${\rm{argmin}}~f \neq \emptyset$;
        \item the problem is differentiable, in the sense that each $f_i$ is differentiable, with $i \to \nabla f_i(x)$ $\mathcal{D}$-measurable, and so is $f$, with $\nabla f(x) = \Exp{\nabla f_i(x)}$, for every $x \in \mathcal{H}$.
    \end{enumerate}
\end{problem}

Within the context of Problem \ref{Pb:main} the \eqref{S1::eq:SGD} algorithm is well defined. 
To analyze its performance, we will further assume the problem to be convex and smooth.
We define a setting combining both convexity and strong convexity, which allows us to unify parts of the analyses. 

\begin{assumption}[Smoothness and convexity]\label{Ass:smooth and mu convex}\label{Ass:convex smooth}\label{Ass:strongly convex smooth}
    Considering Problem \ref{Pb:main}, we assume that 
    \begin{itemize}
        \item there exists $L \in (0, + \infty)$ such that each function   
        $f_i$ is $L$-smooth, i.e.
        $\nabla f_i\colon \mathcal H \to \mathcal H$ is $L$-Lipschitz continuous;
        \item there exists $\mu \in [0, +\infty)$ such that each  function $f_i$ is $\mu$-strongly convex. When $\mu=0$, this means that $f_i$ is convex. 
    \end{itemize}
\end{assumption}

Smoothness and strong convexity can be characterized by means of a variational inequality, which will be central to the derivation of our results.

\begin{lemma}[Expected cocoercivity]\label{L:characterisation expected strong cocoercivity}\label{L:characterisation expected cocoercivity}
    	Let Assumption \ref{Ass:smooth and mu convex} hold.
    	Then, for every $(x,y) \in \mathcal H^2$,
	\begin{align}
    	\label{D:expected strong cocoercivity}  \tag{EC}
    	\frac{1}{2} \Exp{\Vert \nabla f_i(y) - \nabla f_i(x) \Vert^2} + \frac{\mu L}{2} \Vert y-x \Vert^2 
	\leq
    	(L- \mu) \left( f(y) - f(x) \right) +  \langle \mu \nabla f(y) - L \nabla f(x), y-x \rangle.
	\end{align}
\end{lemma}
\begin{proof}
    	Because each $f_i$ is $L$-smooth and $\mu$-strongly convex, we can use \citep[Theorem 4]{taylor_smooth_2017} and take an expectation over $i$.
\end{proof}

We now turn to the constant $\sigma_*^2$ which controls the complexity of the SGD algorithm.

\begin{assumption}[Finite solution variance]\label{Ass:bounded solution variance}
    	Considering Problem \ref{Pb:main}, we assume that the variance at the solution exists, namely that $\sigma^2_* \coloneqq \inf\limits_{x_* \in {\rm{argmin}}~f}\mathbb{E}\left[ \Vert \nabla f_i(x_*) \Vert^2 \right] < + \infty$.
\end{assumption}

We note that, under Assumption \ref{Ass:strongly convex smooth}, the variance term $\Exp{\Vert \nabla f_i(x_*) \Vert^2}$ does not depend on the choice of $x_* \in\argmin f$, as was already observed in \citep[Lemma 4.17]{garrigos_handbook_2024}.

\begin{lemma}
    	Let Assumption \ref{Ass:strongly convex smooth} hold. Then $\Exp{\Vert \nabla f_i(x_*) \Vert^2}$ is constant over ${\rm{argmin}}~f$.
\end{lemma}

\begin{proof}
    Let $x_*,x_*'$ be two minimizers of $f$.
    From Lemma \ref{L:characterisation expected cocoercivity} we know that \eqref{D:expected strong cocoercivity} holds true for the couple $(x_*,x_*')$, which means that
\begin{eqnarray*}
    \frac{1}{2} \Exp{\Vert \nabla f_i(x_*') - \nabla f_i(x_*) \Vert^2} 
    \leq 
    (L- \mu) \left( f(x_*') - f(x_*) \right) +  \langle \mu \nabla f(x_*') - L \nabla f(x_*), x_*'-x_* \rangle.
\end{eqnarray*}
    Because $f(x_*')=f(x_*)=\inf f$ and $\nabla f(x_*') = \nabla f(x_*)=0$, we deduce that
    \begin{equation*}
        \Exp{\Vert \nabla f_i(x_*') - \nabla f_i(x_*) \Vert^2} = 0.
    \end{equation*}
    Then, almost surely with respect to the distribution $\mathcal{D}$ over the indices $i \in \mathcal{I}$, we have $\Vert \nabla f_i(x_*') - \nabla f_i(x_*) \Vert^2 = 0$.
    This means that $\nabla f_i(x_*') = \nabla f_i(x_*)$ almost surely, from which we conclude that 
	$
        \Exp{\Vert \nabla f_i(x_*') \Vert^2}
        =
        \Exp{\Vert \nabla f_i(x_*) \Vert^2}
	$.
\end{proof}

Assuming that the variance at the solution $\sigma_*^2$ exists is trivially satisfied in the case of the minimization of a finite sum.

\begin{lemma}
    If Problem \ref{Pb:main} treats a finite sum of functions, i.e. $\mathcal I$ is finite, then Assumption \ref{Ass:bounded solution variance} holds true.
\end{lemma}

Even for true expectation-minimization problems, where $\mathcal{I}$ is infinite, Assumption \ref{Ass:bounded solution variance} is very mild.
In particular, this assumption is automatically verified in problems where the loss functions $f_i$ are \emph{nonnegative}, which is standard for problems arising in inverse problems and machine learning.

\begin{lemma}
Consider Problem \ref{Pb:main} and assume that the functions $f_i$ are $L$-smooth and are bounded from below uniformly: there exists $r \in \mathbb{R}$ such that $f_i(x) \geq r$ for every $i \in \mathcal I$ and $x \in \mathcal H$.
Then Assumption \ref{Ass:bounded solution variance} is true, with moreover the information that
\begin{equation*}
     \sigma_*^2 \leq 2L (\inf f - r).
\end{equation*}
\end{lemma}

\begin{proof}
Since each $f_i$ is $L$-smooth, we can write \citep[Lemma 2.28]{garrigos_handbook_2024}, for all $x \in \mathcal H$,
    \begin{equation*}
        \Vert \nabla f_i(x) \Vert^2 
        \leq 2L(f_i(x) - \inf f_i) 
        \leq 2L (f_i(x) - r).
    \end{equation*}
    Setting $x = x_* \in {\rm{argmin}}~f$ and taking the expectation of this inequality, we obtain
    \begin{equation*}
        \Exp{\Vert \nabla f_i(x_*) \Vert^2} \leq 2L (f(x_*) - r) = 2L (\inf f - r) < + \infty,
    \end{equation*}
    which concludes since $\sigma_*^2=\Exp{\Vert \nabla f_i(x_*) \Vert^2} $.
\end{proof}

We now highlight the connection between the variance at the solution $\sigma^2_*$ with the notion of interpolation.
The following facts can be found in \cite{garrigos_handbook_2024}, which we have extended here from a finite-sum problem to our expectation-minimization Problem \ref{Pb:main}.

\begin{definition}[Interpolation]\label{D:interpolation}
    Consider Problem \ref{Pb:main}.
    We say that interpolation holds for the family of functions $\{f_i\}_{i \in \mathcal{I}}$ if 
    there exists $x_*$ which belongs to ${\rm{argmin}}~f_i$ for $\mathcal{D}$-almost every $i$.
\end{definition}

\begin{lemma}
\label{L:interpolation argmin and inf coincide}
    In the context of Problem \ref{Pb:main}, if interpolation holds, then $\Exp{\inf f_i} = \inf f$ and $\cap_{i \in \mathcal{I}} \ {\rm{argmin}}~f_i = {\rm{argmin}}~f$, where the intersection is to be understood for $\mathcal{D}$-almost every $i \in \mathcal{I}$.
\end{lemma}

\begin{proof}
    Because interpolation holds, we may select $x_* \in \cap_{i \in \mathcal{I}} \ {\rm{argmin}}~f_i$.
    Let us start by showing that $x_* \in {\rm{argmin}}~f$.
    For this, consider any $x \in \mathcal{H}$ and write
    \begin{equation*}
        f(x_*) = \Exp{ f_i(x_*)} = \Exp{\inf f_i} \leq \Exp{f_i(x)} = f(x).
    \end{equation*}
    This proves that $x_* \in {\rm{argmin}}~f$.
    Now let us prove the first part, by observing that
    \begin{equation*}
        \inf f = f(x_*) = \Exp{f_i(x_*)} = \Exp{\inf f_i}.
    \end{equation*}
    Secondly, we prove the second point, for which we are only left with the reverse inclusion.
    Consider some $x \in {\rm{argmin}}~f$.
    Then 
    \begin{equation*}
        f(x) = \inf f = \Exp{ \inf f_i} \ \implies \ \Exp{f_i(x) - \inf f_i} = 0.
    \end{equation*}
    Since $f_i(x) - \inf f_i \geq 0$ for all $i$, it must hold that $f_i(x) - \inf f_i = 0$ for almost every $i \in \mathcal{I}$, which proves the claim.
\end{proof}

\begin{lemma}
\label{L:interpolation and sigma}
    In the context of Problem \ref{Pb:main}, if interpolation holds, then $\sigma_*^2 = 0$. This becomes an equivalence if Assumption \ref{Ass:smooth and mu convex} holds.
\end{lemma}

\begin{proof}
    If interpolation holds, then for every $x_* \in {\rm{argmin}}~f$ we have that $x_* \in {\rm{argmin}}~f_i$ for almost every $i \in \mathcal{I}$.
    For such $i \in \mathcal{I}$, the optimality condition gives $\nabla f_i(x_*)=0$, which after taking the norm and expectation leads to $\Exp{\Vert \nabla f_i(x_*) \Vert^2} = 0$.

    Reciprocally, if we assume that $\sigma_*^2 = 0$, it means that there exists $x_* \in {\rm{argmin}}~f$ such that $\Exp{\Vert \nabla f_i(x_*) \Vert^2} = 0$.
    This is an expectation of nonnegative terms, so we deduce that $\nabla f_i(x_*) = 0$ for almost every $i \in \mathcal{I}$.
    Exploiting our convexity assumption, we deduce that $x_*$ is a minimizer of $f_i$ for almost every $i \in \mathcal{I}$, which shows that interpolation holds.
\end{proof}


\section{Reducing a Lyapunov Analysis to the Resolution of a System of Inequalities}\label{S3::sec__}

Our main approach to derive bounds for SGD 
is to rely on an analysis of a new parametrized family of Lyapunov energies, which we presented in the introduction and recall below.
After introducing these Lyapunov energies, we will show how to deduce bounds for SGD from a Lyapunov decrease in Section \ref{sec:bounds-decrease}.
Next, we present our main technical result which states that obtaining such Lyapunov decrease can be reduced to solving a system of inequalities (see Section \ref{sec:decrease-ineq}).

\subsection{Bounds for SGD from a Lyapunov Decrease}\label{sec:bounds-decrease}

Let $T \geq 1$ be fixed, let $x_* \in {\rm{argmin}}~f$, and let $(x_t)_{t=0}^{T}$ be generated by the SGD algorithm for some fixed step-size $\gamma >0$.
Given a set of parameters $\rho, a_0, \dots, a_T, e_0, \dots, e_{T-1} \ge 0$, we define the following Lyapunov energy, for $t = 0, \dots, T$: 
\begin{equation}\label{S0::eq:Lyapunov}\tag{Lyapunov}
    E_t \coloneqq  
    a_t \Vert x_t - x_* \Vert^2
    + \rho \sum_{s=0}^{t-1} \left( f(x_s) - \min f \right)
    - \sum_{s=0}^{t-1} e_s\sigma_*^2,
\end{equation}
where, by convention, the empty sum $\sum_{s=0}^{-1}$ is equal to zero.
A first result consists in deriving bounds on the average function value gap, provided that the Lyapunov parameter $\rho$ is nonzero.
This bound will typically be obtained for convex smooth problems. The proof is based on a telescopic sum and an application of Jensen's inequality, as is often the case for convex smooth objectives, see for instance in \cite{taylor_stochastic_2019,gower_sgd_2021}.

\begin{lemma}[Bound from Lyapunov decrease. Convex case]\label{L:SGD bound from lyapunov decrease}
    Consider Problem \ref{Pb:main}, assume that $f$ is convex, and let $T \geq 1$.
    Assume that $\Exp{E_{t+1}} \leq \Exp{E_t}$ for every $t=0, \dots, T-1$, and that $\rho > 0$. 
    Then, if we denote by $\bar e = \tfrac{1}{T}\sum_{t=0}^{T-1} e_t$ and $\bar x_T = \tfrac{1}{T}\sum_{t=0}^{T-1} x_t$,
    \begin{equation*}
        \mathbb{E} \left[ f(\bar x_T) - \min f \right]
        \leq 
        \frac{1}{T}\sum_{t=0}^{T-1}\Exp{f(x_t)-\min f}
        \leq 
        \frac{a_0 \Vert x_0 - x_* \Vert^2}{\rho T} + \frac{\bar e \sigma_*^2}{\rho}.
    \end{equation*}
\end{lemma}

\begin{proof}
    Iterating $\Exp{E_{t+1}} \leq \Exp{E_t}$ for $t=0, \dots, T-1$ yields that $\Exp{E_T}\le E_0$, or, equivalently,
    \[
        a_T\Exp{\|x_T-x_*\|^2}+\rate \sum_{t=0}^{T-1}\Exp{f(x_t)-\min f}-\sum_{t=0}^{T-1}e_t\sigma_*^2\le a_0\|x_0-x_*\|^2.
    \]
    Bounding $a_T\Exp{\|x_T-x_*\|^2}\ge 0$ and dividing by $\rate T>0$ yields
    \[
        \mathbb{E}\left[ f(\bar x_T) - \min f \right]
        \le \frac1T\sum_{t=0}^{T-1}\Exp{f(x_t)-\min f}\le \frac{a_0\|x_0-x_*\|^2}{\rate T}+\frac{\bar e \sigma_*^2}{\rate},
    \]
    where in the first inequality we used Jensen's inequality.
\end{proof}

As a second result, a Lyapunov decrease can yield a bound of the square distance of the iterate $x_t$ to a minimizer, provided the Lyapunov parameter $a_T$ is nonzero.
This is standard for strongly convex problems, as we will see in Section \ref{S3::sec:str}. The proof, also based on a telescopic argument,  is again typical for this class of problems, see for instance \cite{gower_sgd_2019}.

\begin{lemma}[Bound from Lyapunov decrease. Strongly convex case]\label{L:SGD bound from lyapunov decrease strongly convex}
    Consider Problem \ref{Pb:main}, and let $T \geq 1$.
    Assume that $\Exp{E_{t+1}} \leq \Exp{E_t}$ for every $t=0, \dots, T-1$, and that $a_{T} >0$.
    Then, if we denote by $e_T^{sum} = \sum_{t=0}^{T-1} e_t$,
    \begin{equation*}
        \Exp{\Vert x_T - x_* \Vert^2} 
        \leq 
        \frac{a_0 \Vert x_0 - x_* \Vert^2}{a_T} + \frac{ e_T^{sum} \sigma_*^2}{a_T}.
    \end{equation*}
\end{lemma}
\begin{proof}
    Iterating $\Exp{E_{t+1}} \leq \Exp{E_t}$ for $t=0, \ldots, T-1$ yields $\Exp{E_T}\le E_0$, or in other words that
    \[
        a_T\Exp{\|x_T-x_*\|^2}+\rate\sum_{t=0}^{T-1}\Exp{f(x_t)-\min f}\le a_0\|x_0-x_*\|^2 + e_T^{sum}\sigma_*^2.
    \]
    Since $f(x_t) - \min f \geq 0$ for all $t=0, \ldots, T-1$, and $\rho\ge 0$, this yields the wanted inequality after division by $a_T>0$.
\end{proof}

\subsection{Lyapunov Decrease from a System of Inequalities}\label{sec:decrease-ineq}

Lemmas \ref{L:SGD bound from lyapunov decrease} and \ref{L:SGD bound from lyapunov decrease strongly convex} reduce the problem of finding upper bounds for SGD to finding parameters $\rho, a_0, \dots, a_T, e_0, \dots, e_{T-1} \ge 0$ that allow the Lyapunov energy to decrease in expectation.
The next theorem provides sufficient conditions for this to be the case, and will be at the heart of our subsequent results.
Indeed, this will allow to shift the focus from sequences generated by SGD to a system of inequalities on real numbers. 

\begin{theorem}[Sufficient conditions for Lyapunov decrease]\label{T:lyapunov sufficient conditions strongly convex}\label{T:lyapunov sufficient conditions}
    Let Assumptions \ref{Ass:strongly convex smooth} and \ref{Ass:bounded solution variance} hold true, and let $T \geq 1$.
    Assume that there exist parameters $(\alpha_t, \beta_t)_{t=0}^{T-1}$ such that, for every $t=0, \dots, T-1$, the following conditions are verified:
    \begin{enumerate}
        \item\label{T:lyapunov sufficient conditions strongly convex:positivity}\label{T:lyapunov sufficient conditions:positivity}\label{T:Lyap:1}
        $\rho, a_t, e_t, \alpha_t, \beta_t \geq 0$,
        \item\label{T:lyapunov sufficient conditions strongly convex:rho}\label{T:lyapunov sufficient conditions:rho} \label{T:lyapunov sufficient conditions:rho upper bound} \label{T:Lyap:2}
        $\rho \leq 2(L-\mu)(\alpha_t - \beta_t)$,
        \item\label{T:lyapunov sufficient conditions strongly convex:monotone at}\label{T:lyapunov sufficient conditions:monotone at} \label{T:Lyap:3}
        $a_{t+1} \leq \mu L (\alpha_t + \beta_t) +  a_t $, 
        \item\label{T:lyapunov sufficient conditions strongly convex:annoying}\label{T:lyapunov sufficient conditions:annoying} \label{T:Lyap:4}
        $a_{t+1} \gamma^2 \leq \alpha_t + \beta_t$,
        \item\label{T:lyapunov sufficient conditions strongly convex:curious}\label{T:lyapunov sufficient conditions:curious} \label{T:Lyap:5}
        $(a_{t+1}\gamma - \alpha_t L - \beta_t \mu )^2 \leq (\mu L (\alpha_t + \beta_t) + a_t - a_{t+1})(\alpha_t + \beta_t - a_{t+1}\gamma^2)$,
        \item\label{T:lyapunov sufficient conditions strongly convex:variance}\label{T:lyapunov sufficient conditions:variance} \label{T:Lyap:6}
        $a_{t+1} \gamma^2 (\alpha_t+\beta_t) \leq (\alpha_t + \beta_t - a_{t+1}\gamma^2)e_t$.
    \end{enumerate}
    Then $\Exp{E_{t+1}} \leq \Exp{E_t}$ for every $t=0, \ldots, T-1$. 
\end{theorem}

\begin{proof}
    Using the definition of $x_{t+1}$ allows us to write
    \begin{eqnarray}
         E_{t+1} - E_t \label{eq:lyapunov derivative full} 
         =
         \rho \left( f(x_t) - \min f \right)
         - e_t \sigma_*^2
         +
        (a_{t+1}  - a_t) \Vert x_t - x_* \Vert^2 
        + a_{t+1}\gamma^2 \Vert \nabla f_i(x_t) \Vert^2 - 2 a_{t+1}\gamma \langle \nabla f_i(x_t), x_t - x_* \rangle.
    \end{eqnarray}
    In what follows, we denote by $\mathbb{E}_t \coloneqq  \Exp{ \cdot \ | \ x_0, \dots, x_t }$ the conditional expectation with respect to the iterates up to and including $x_t$.
    Multiplying the expected cocoercivity inequalities (Lemma \ref{L:characterisation expected cocoercivity}) for the pair $(x_t, x_*)$ by $2\alpha_t \geq 0$ and for the pair $(x_*, x_t)$ by $2\beta_t \geq 0$, and summing them gives
    \begin{eqnarray*}
        && ({\alpha_t + \beta_t}) \Expt{\Vert \nabla f_i(x_t) - \nabla f_i(x_*) \Vert^2} + {\mu L(\alpha_t + \beta_t)} \Vert x_t-x_* \Vert^2 \\
        &\leq &
        2(L- \mu)(\beta_t - \alpha_t) \left(  f(x_t) - \min f \right) + 2(\alpha_t L + \beta_t \mu ) \langle  \nabla f(x_t), x_t-x_* \rangle.
    \end{eqnarray*}
    Adding $\Expt{E_{t+1}} - E_t$ on both sides and using the expression obtained in \eqref{eq:lyapunov derivative full}, taking the expectation and developing the squares yields
        \begin{eqnarray*}
        \Exp{E_{t+1} - E_t} 
            &\leq &
            (\rho + 2(L-\mu)(\beta_t - \alpha_t)) \mathbb{E}\left[ f(x_t) - \min f\right] \\
            && - \Exp{A {\Vert x_t - x_* \Vert^2} + B {\Vert \nabla f_i(x_t) \Vert^2 } + B' {\Vert \nabla f_i(x_*) \Vert^2}
             + 
            2C {\langle \nabla f_i(x_t), x_t - x_* \rangle
            } -
            2 D {\langle \nabla f_i(x_t), \nabla f_i(x_*) \rangle
            }},
        \end{eqnarray*}
        where we simplified the expression by introducing the following constants
        \begin{equation*}
        \begin{cases}
            A &= \mu L (\alpha_t + \beta_t ) + a_t - a_{t+1}, \\
            B &= \alpha_t + \beta_t - a_{t+1}\gamma^2, \\
            B' &= e_t + \alpha_t + \beta_t, \\
            C &= a_{t+1}\gamma - \alpha_t L - \beta_t \mu, \\
            D &= \alpha_t+\beta_t.
        \end{cases}
    \end{equation*}
    Note that each of the above expectations are well defined, see Proposition \ref{S1::prop:finite expectations} in the appendix for the details.
    For the Lyapunov energy to be decreasing in expectation, it is enough that the right-hand side is nonpositive.
    For the first term, because $f(x_t) - \min f \geq 0$, it is enough to assume that 
    $\rho \leq 2(L-\mu)(\alpha_t - \beta_t)$, which corresponds to Condition \ref{T:lyapunov sufficient conditions strongly convex:rho}.
    The remaining terms are simply the expectation of a quadratic polynomial in $X = x_t - x_*$, $Y = \nabla f_i(x_t)$ and $Y' = \nabla f_i(x_*)$.
    Using elementary linear algebra, it can be shown (see Lemma \ref{L:polynomial quadratic nonnegative conditions} below) that for this polynomial to be nonnegative, it is enough to require
    \begin{equation*}
        A, B, B' \ge 0, \quad C^2\le AB, \quad \text{and}\quad D^2\le BB'.
    \end{equation*}
    A simple calculation shows that these conditions correspond exactly to the remaining conditions of our theorem. 
\end{proof}

\begin{remark}
    We note that Theorem \ref{T:lyapunov sufficient conditions}, which is the key theorem for all of our subsequent results, is valid provided that the inequality of Lemma \ref{L:characterisation expected cocoercivity} holds for any pair of points $(x, x_*)$ and $(x_*, x)$, where $x\in \mathcal H$. As such, we could replace Assumption \ref{Ass:strongly convex smooth} by such a weaker assumption. 
    Moreover, in the case where $\beta_t\equiv 0$, it is only necessary for the conclusion of Lemma \ref{L:characterisation expected cocoercivity} to hold for any pair $(x, x_*)$, where $x\in \mathcal H$. 
    Although such nonconvex extension is of interest \cite{gower_sgd_2021}, 
    we do not delve further into it in this work. 
\end{remark}

We end this section with the technical lemma that we used in the last proof.

\begin{lemma}[Nonnegative quadratic polynomial]\label{L:polynomial quadratic nonnegative conditions}
    Let $A,B,B',C,D \in \mathbb{R}$ be such that 
    \begin{equation}\label{SA::eq:conds}
        A, B, B' \geq 0,\quad C^2 \leq AB, \quad \text{and}\quad  D^2 \leq  BB'.
    \end{equation}
    Let $X,Y,$ and $Y'$ be three random variables over $\mathcal H$ such that $\mathbb{E}\left[\langle X, Y'\rangle \right] = 0$.
    Then
    \begin{eqnarray*}
        \Exp{
        A  \Vert X \Vert^2
        + 
        B \Vert Y \Vert^2
        + 
        B'  \Vert Y' \Vert^2
        +
        2C  \langle Y, X \rangle
        -
        2D  \langle Y, Y' \rangle 
        }
        \geq 0.
    \end{eqnarray*}
\end{lemma}
    \begin{proof}
    If $B=0$, then $C, D=0$, and the result follows. Therefore, suppose $B>0$.
    
    Let us note $P$ to be the quantity that we want to see nonnegative.
    We start by exploiting the fact that $\Exp{\langle X,Y' \rangle}=0$ to write that $P = P + \theta \Exp{\langle X,Y' \rangle }$, for every $\theta \in \mathbb{R}$.
    Now, see that for $P \geq 0$ to be true, it is enough for
    \begin{equation*}
        \hat P_\theta(x,y,y') = A  \Vert x \Vert^2
        + 
        B  \Vert y \Vert^2
        + 
        B ' \Vert y' \Vert^2
        +
        2C \langle y, x \rangle
        -
        2D  \langle y, y' \rangle
        + 2 \theta  \langle x,y' \rangle
    \end{equation*}
    to be nonnegative for every $x,y,y' \in \mathcal H$ and for some $\theta  \in \mathbb{R}$.
    Our polynomial $\hat P_\theta$ is equal to $\langle M_\theta z,z \rangle$ where $z=(x,y,y')$ and 
    \begin{equation*}
        M_\theta = 
        \begin{pmatrix}
        A  & C & \theta  \\
        C  & B  & - D \\
        \theta  & - D & B '
        \end{pmatrix}.
    \end{equation*}
    As such, $\hat P_\theta\ge 0$ if, and only if, $M_\theta\succeq 0$, which, by Sylvester's criterion, is equivalent to
    \begin{subequations}\label{SA::eq:M}
        \begin{align}
            & A, B, B' \geq 0, \label{SA::eq:M1}\\
            & C^2\le AB, \quad \theta^2 \le AB', \quad D^2\le BB', \label{SA::eq:M2} \\
            & A (BB' - D^2) \geq B'C^2 + B\theta^2 + 2CD\theta.  \label{SA::eq:M3}
        \end{align}
    \end{subequations}
    So $P \geq 0$ if there exists $\theta  \in \mathbb{R}$ such that Inequalities \eqref{SA::eq:M} are verified.
    We can already see Equations \eqref{SA::eq:M1}-\eqref{SA::eq:M2} imply Equations \eqref{SA::eq:conds}.
    Let us prove that in fact the existence of a $\theta\in \R$ satisfying Equations \eqref{SA::eq:M} is equivalent to Equations \eqref{SA::eq:conds}.
    
    To prove equivalence, we assume Equations \eqref{SA::eq:conds} and show that there exists a $\theta\in \R$ such that
    \begin{equation}\label{SA::eq:exists}
        AB' \geq \theta ^2 \quad \text{ and } \quad A (BB' - D^2) \geq B'C^2  + B\theta ^2 + 2 CD\theta.
    \end{equation}
    The last inequality is quadratic in $\theta$, and has solutions if, and only if, its discriminant $\Delta$ is nonnegative. A straightforward computation shows that
    \[
        \Delta = 4(BB' - D^2)(AB-C^2),
    \]
    which guarantees that $\Delta \ge 0$ given Equations \eqref{SA::eq:conds}. The quadratic may hence be rewritten as
    \begin{equation*}
        \frac{-2CD - \sqrt{\Delta}}{2B}
        \leq \theta  \leq \frac{-2CD + \sqrt{\Delta}}{2B}.
    \end{equation*}
    For \eqref{SA::eq:exists} to have a solution, the above must hold simultaneously with $-\sqrt{AB'} \leq \theta \leq \sqrt{AB'}$.
    This is equivalent to the intersection of two intervals to be nonempty, which is the case if
    \begin{equation*}
        \begin{cases}
            \frac{-2CD + \sqrt{\Delta}}{2B} \ge -\sqrt{AB'} \quad &\text{if $CD\ge 0$}, \\
            \frac{-2CD - \sqrt{\Delta}}{2B} \le \sqrt{AB'} \quad &\text{if $CD\le 0$}.
        \end{cases}
    \end{equation*}
    or equivalently,
    \begin{eqnarray*}
        2|CD| \leq 2 B \sqrt{AB'} + \sqrt{\Delta}.
    \end{eqnarray*}
    As all terms are nonnegative, this inequality is equivalent to
    \begin{equation*}
        4 C^2D^2 \leq 4 B^2 AB' + 4 B \sqrt{AB' \Delta} + \Delta,
    \end{equation*}
    which holds true since $D^2 \leq B B'$, $C^2 \leq AB$ and $\Delta\ge 0$.
\end{proof}

\section{Bias-Optimal Bounds in the Convex Smooth Setting}\label{S2::sec}\label{S2::sec:res}

We now present bounds for SGD that follow directly from Lemma \ref{L:SGD bound from lyapunov decrease}.
These bounds are obtained by identifying specific Lyapunov parameters that satisfy the sufficient conditions of Theorem \ref{T:lyapunov sufficient conditions}, with the objective of maximizing the value of $\rate$, which controls the bias term in Lemma \ref{L:SGD bound from lyapunov decrease}.
We note that the Lyapunov parameters $e_t$ arising from these proofs are at times rather intricate. For clarity of exposition, we therefore present bounds based on simplified parameter choices in this section. The complete, more detailed bounds, as well as the full proofs and the underlying motivation for the parameter selection, are deferred to Appendix \ref{SA::sec:conv_proof}.

The presentation of our results is split into three parts based on the normalized step-size $\gamma L$, namely Theorem \ref{S2::thm:main_small} for short step-sizes $\gamma L\in (0, 1)$, Theorem \ref{S2::thm:main_1} for the critical step-size $\gamma L=1$, and Theorem \ref{S2::thm:main_large} for large step-sizes $\gamma L\in (1, 2)$. 

\subsection{Short Step-Sizes $\gamma L <1$}\label{S3b::ssec:short}

\begin{theorem}[Convex case, short step-sizes]\label{S2::thm:main_small}
    Let Assumptions \ref{Ass:strongly convex smooth} and \ref{Ass:bounded solution variance} hold true with $\mu=0$,
    and
    let $(x_t)$ be generated by \eqref{S1::eq:SGD} with $\gamma L \in (0,1)$.
    Then, for every $T \geq 1$, with $\bar x_T = \tfrac{1}{T}\sum_{t=0}^{T-1} x_t$,
    \begin{equation*}
        \Exp{f(\bar x_T) - \min f }
        \leq 
        \frac{1}{T}\sum_{t=0}^{T-1} \mathbb{E} \left[ f(x_t) - \min f \right]
        \leq 
        \frac{L \Vert x_0 - x_* \Vert^2}{2 \gamma L T + 2(1 - \gamma L)} + \frac{\gamma \sigma_*^2}{2(1 - \gamma L)}.
    \end{equation*}
    Moreover, the above bound is bias-optimal, in the sense that there exists a problem for which $\sigma_*^2=0$ and both inequalities are equalities.
\end{theorem}

\begin{proof}
    The first inequality follows from Jensen's inequality.
    The second follows from Lemma \ref{L:SGD bound from lyapunov decrease} and Theorem \ref{T:lyapunov sufficient conditions strongly convex}, where in the latter we select the following parameters:
    \[
        \rho = 2 \gamma + \frac{2(1 - \gamma L)}{LT}, \quad a_t = \frac{T-t}{T} \frac{1+ \gamma L (T-1) }{1 + \gamma L(T-t-1)}, \quad 
        e_t = \frac{\gamma^2 }{1-\gamma L}, \quad 
        \alpha_t \equiv \alpha =\frac{\rho}{2L}, \quad \beta_t \equiv \beta =0.
    \]
    A detailed proof of the validity of those parameters, including the motivation for their choice, is available in Appendix \ref{SA::sec:proof_short}.
    Regarding the bias-optimality, we consider the convex and $L$-smooth Huber function, parameterized by $\eta>0$,
    \begin{align*}
    \mathcal{H}_{\eta}(x) = \begin{cases}
        \eta L \|x\| - \frac{L}{2}\eta^2 &\text{if} \quad \|x\| > \eta \\
        \frac{L}{2}\|x\|^2 &\text{if} \quad \|x\| \leq \eta.
    \end{cases}
    \end{align*}
    When $\eta$ is sufficiently small, namely $\eta = \frac{\Vert x_0 \Vert}{1+(T-1)\gamma L}$, it can be shown that the first $T$ iterates remain within the region where $\mathcal H_\eta$ is linear. In this setting, a direct computation yields
    \[
    \mathcal{H}_{\eta}(\overline{x}_T) - \inf \mathcal{H}_{\eta} = \frac{L \|x_0-x^*\|^2}{2 (1+(T-1)\gamma L)},
    \]
    which matches the bias term of Theorem \ref{S2::thm:main_small}. The full argument is provided in the proof of Proposition \ref{SC::prop:GD}. 
    \end{proof}

\begin{remark}[Related works]
    Theorem \ref{S2::thm:main_small} provides a better bound, both in terms of bias and variance, than what can be found in the literature for vanilla SGD, while also extending the range of step-sizes for which the bounds apply.
    For instance, \cite[Theorem D.6]{gower_sgd_2021} requires $\gamma L\in (0, \frac{1}{2})$ and derives the bound
        \[
           f(\bar x_T)-\inf f \le \frac{\|x_0-x_*\|^2}{2\gamma (1-2\gamma L)T}+\frac{\gamma}{1-2\gamma L} \sigma_*^2.
        \]
        Note that the bias term blows up when $\gamma L$ approaches the upper bound $\tfrac{1}{2}$, which is not the case for our bound when $\gamma L \to 1$.
        However, they rely on weaker assumptions than ours, namely the smoothness of $f$ (instead of each $f_i$) and the quasar-convexity of the $f_i$ (instead of the convexity), and cover the use of variable step-size, which we do not address in this work, see also \cite[Theorem 5.3]{garrigos_handbook_2024}.
    A less straightforward comparison can be made with \cite[Theorem 15]{taylor_stochastic_2019}, where the authors investigate a primal averaging variant of SGD, which at every iteration evaluates the gradient at $\bar x_t$ instead of $x_t$. For $\gamma L\in (0, 1)$, they establish the upper bound
        \[
            \frac{\|x_0-x_*\|^2}{ 2\gamma T}+\frac{\gamma}{2(1-\gamma L)} \sigma_*^2.
        \]
    As our bound is strictly better, although asymptotically equivalent, we see that we do not need to rely on primal averaging to obtain such bound.
\end{remark}

\begin{remark}[About the bounded variance assumptions]
    Our results indicate that the commonly imposed assumptions of uniformly bounded variance or uniformly bounded gradients are {not necessary} to improve the complexity guarantees for SGD. 
    Indeed, the worst-case function arising in the proof of Theorem \ref{S2::thm:main_small} has bounded gradients, meaning that the bias term cannot be improved by assuming gradient boundedness. 
    This also means that Theorem \ref{S2::thm:main_small} is bias-optimal for SGD under the uniformly bounded variance assumption.
    We for instance note that our bound improves the one obtained under that assumption in \cite[Theorem 6]{taylor_stochastic_2019}, whose bias term is strictly larger than ours.
\end{remark}

\subsection{Critical Step-Size $\gamma L =1$}\label{sec::crit_step_size}

We now turn to the case $\gamma L=1$, which we call the \textit{critical} step-size.
While standard for GD, this case has surprisingly not been included in any complexity rates for SGD without bounded variance assumptions. 
The only result we are aware of is \cite[Theorem 3]{DOLoiLarMit23a}, which provides  a bound on the averaged gradient norm $\tfrac{1}{T} \sum_{t=0}^{T-1} \Vert \nabla f(x_t) \Vert^2$ when $\gamma L \leq 1$, for a more general algorithm (the stochastic mirror descent).
This theorem is not a complexity bound, because its variance term is independent of $\gamma$ and is essentially equal to $\sigma_*^2$ (see \cite[Lemma 4.18.1]{garrigos_handbook_2024}).
We show next that our complexity bounds obtained for the expected optimality gap when $\gamma L <1$ can be extended to the critical case $\gamma L= 1$.
As far as we know, this is the first result provided in this setting for this metric.

Considering the result of Theorem \ref{S2::thm:main_small} for short step-sizes and letting $\gamma L \to 1$, we note that the bias term tends to $\tfrac{1}{2 \gamma T}$
while the variance term diverges since its denominator is a multiple of $1 - \gamma L$.
In the deterministic case, we know that this bias term $\tfrac{1}{2 \gamma T} = \tfrac{L}{2T}$ can be achieved, and that it is the best possible worst-case rate one can have when considering the average function value metric (see Corollary \ref{C:GD optimal rates} in the appendix for details).
This raises the question whether it is possible to obtain a bound for SGD whose bias term matches the best worst-case rate of GD. Our next result answers this negatively.

\begin{proposition}[No optimal bias at critical step-size, convex case]\label{thm:sing_cvx}
    Let $L>0$ and $T\ge 2$, and let $\gamma_{\crit}=\frac1L$. There exists no finite constant $V$, depending only on $L$ and $T$, such that 
    \[
        \frac1T\sum_{t=0}^{T-1}\Exp{f(x_t)-\min f}\le \frac{L\|x_0-x_*\|^2}{2T}+V\sigma_*^2
    \]
    holds true for every $x_0\in \H$ and every problem satisfying Assumptions \ref{Ass:strongly convex smooth} and \ref{Ass:bounded solution variance} for $\mu=0$, where $(x_t)$ is generated through Algorithm \eqref{S1::eq:SGD} with $\gamma=\gamma_{\crit}$ and starting from $x_0$.
\end{proposition}
\begin{proof}
    To show this, we consider a family of problems, whose limiting behavior contradicts the existence of a finite $V$. See details in Appendix \ref{SA3::ssec:optimal}.
\end{proof}

There is no free lunch: using a stochastic estimator of the gradient comes at a price, that is, one cannot hope for a bound for SGD which decomposes as the best rate of GD plus a multiple of the solution variance $\sigma^2_*$.
However, we can get arbitrarily close.
First, we prove general bounds with any bias strictly larger than $\tfrac{1}{2 \gamma T}$.
Second, we show that when interpolation holds, SGD behaves like GD, and enjoys the same rate of convergence.

\begin{theorem}[Convex case, critical step-size]\label{S2::thm:main_1}
    Let Assumptions \ref{Ass:strongly convex smooth} and \ref{Ass:bounded solution variance} hold true with $\mu=0$.
    Let $(x_t)$ be generated by \eqref{S1::eq:SGD} with $\gamma L =1$.
    Then, for every  $\varepsilon \in (0,1)$ and $T \geq 1$, with $\bar x_T = \tfrac{1}{T}\sum_{t=0}^{T-1} x_t$,
    \begin{equation*}
        \mathbb{E} \left[ f(\bar x_T) - \min f \right]
        \leq 
        \frac{1}{T}\sum_{t=0}^{T-1} \mathbb{E} \left[ f(x_t) - \min f \right]
        \leq 
        \frac{\Vert x_0 - x_* \Vert^2}{2(1-\varepsilon) \gamma T} + \frac{1+ \varepsilon}{2\varepsilon(1-\varepsilon)} \gamma \sigma_*^2 .
    \end{equation*}
\end{theorem}
\begin{proof}
    It suffices to select the following parameters in Theorem \ref{T:lyapunov sufficient conditions strongly convex}:
    \[
        \rho = 2(1-\varepsilon) \gamma, \quad a_t \equiv a = 1,\quad \alpha_t = \gamma^2,\quad  \beta_t =  \gamma^2 \varepsilon, \quad \text{and}\quad e_t = \gamma^2 \cdot \frac{1+\varepsilon}{\varepsilon}.
    \]
    A detailed proof of their validity is presented in the proof of Theorem \ref{T:convex large stepsize subopt:appendix}.
\end{proof}

\begin{remark}[Stochastic proximal algorithm]\label{S2::rem:prox}
    While convergence rates for the deterministic proximal algorithm are well-studied, complexity results for its stochastic counterpart are unknown in a general setting.
    A classical trick allows to rewrite the stochastic proximal algorithm as a particular instance of SGD with a step-size of $\gamma L =1$.
    Theorem \ref{S2::thm:main_1} then enables us to derive the first complexity bound for this algorithm with no other assumption on the functions than convexity, avoiding, for instance, the need to impose finiteness and Lipschitzness as in \cite{davis_stochastic_2019}. 
    We develop this idea in Section~\ref{SA::sec:prox}.
\end{remark}

Naturally, when interpolation holds (i.e., when $\sigma_{*} = 0$) the variance term vanishes, allowing us to take the limit $\varepsilon \to 0$ and thereby obtain an optimal convergence rate. 
We note that an identical rate was previously established in \cite[Theorem 8]{taylor_stochastic_2019} for a modified version of SGD incorporating a primal averaging scheme.

\begin{corollary}[Convex case with interpolation, critical step-size]\label{S2::coro:inter_1}
    Let Assumptions \ref{Ass:strongly convex smooth} and \ref{Ass:bounded solution variance} hold true, with $\mu=0$,
    and assume that $\sigma_*^2=0$.
    Let $(x_t)$ be generated by \eqref{S1::eq:SGD} with $\gamma L =1$.
    Then, for every $T \geq 1$, with $\bar x_T = \tfrac{1}{T}\sum_{t=0}^{T-1} x_t$,
    \[
        \Exp{f(\bar x_T)-\min f}\le 
        \frac{1}{T}\sum_{t=0}^{T-1} \mathbb{E} \left[ f(x_t) - \min f \right]
        \leq 
        \frac{ \|x_0-x_*\|^2}{2 \gamma T}.
    \]
    Moreover, the above bound is optimal, in the sense that there exists a problem for which $\sigma^2_*=0$ and both inequalities are  equalities.
\end{corollary}

\begin{proof}
    Apply Theorem \ref{S2::thm:main_1} with $\varepsilon \to 0$ to get the bound.
    The bias-optimality is again achieved through the same Huber function as for $\gamma L\in (0,1)$, see the proof of Proposition \ref{SC::prop:GD}.
\end{proof}

\subsection{Large Step-Sizes $\gamma L >1$}

The case $\gamma L\in (1, 2)$, referred to as \textit{large step-sizes}, is yet unexplored for SGD without bounded variance assumptions. We present below the first result of this type. 

\begin{theorem}[Convex case, large step-sizes, optimal bias]\label{S2::thm:main_large}
    Let Assumptions \ref{Ass:strongly convex smooth} and \ref{Ass:bounded solution variance} hold true, with $\mu=0$,
    and
    let $(x_t)$ be generated by \eqref{S1::eq:SGD} with $\gamma L \in (1,2)$.
    Then, for every $T \geq 1$, with $\bar x_T = \tfrac{1}{T}\sum_{t=0}^{T-1} x_t$,
    \begin{equation*}
        \Exp{f(\bar x_T) - \min f}
        \leq 
        \frac{1}{T}\sum_{t=0}^{T-1} \mathbb{E} \left[ f(x_t) - \min f \right]
        \leq
        \frac{\delta_T \Vert x_0 - x_* \Vert^2 }{2 \gamma (2 - \gamma L) T} + \frac{\gamma \sigma_*^2}{2(2 - \gamma L)^3}
        \frac{1}{T (\gamma L -1)^{2T-2}},
    \end{equation*}
    where $\delta_T = 1 - (\gamma L-1)^{2T} \in (0,1)$.
    Moreover, this bound is bias-optimal, in the sense that there exists a problem for which $\sigma_*^2=0$ and the second inequality is an equality.
\end{theorem}
\begin{proof}
    It suffices to select the following parameters in Theorem \ref{T:lyapunov sufficient conditions strongly convex}, with $\theta = (\gamma L - 1)^2$:
    \[
        \rho = \frac{2 \gamma(2 - \gamma L)}{1 - \theta^T},\quad a_t = \frac{1 - \theta^{T-t}}{1- \theta^T}, \quad \beta_t = \frac{ \gamma(\gamma L -1)(1 - \theta^{T-t-1})}{L(1-\theta^T)}, \quad \alpha_t=\beta_t+\frac{\rho}{2L},
        \]
        \[
        e_t = 
        \frac{\gamma^2}{2-\gamma L} 
        \frac{1 - \theta^{T-t-1}}{1- \theta^T} 
        \left( \frac{\gamma L }{\theta^{T-t-1}} - 2(\gamma L -1) \right), 
        \quad
        \bar{e}_T = \frac{1}{T}\sum_{t=0}^{T-1} e_t.
    \]
    A detailed proof of their validity, including the motivation for their choice, is presented in Appendix \ref{SA::sec:proof_large}.
    Proving that the bias term is optimal can be done by considering $f(x) = \tfrac{L}{2}\Vert x \Vert^2$, see Proposition \ref{SD::sec:proof_optimality_rates} in the appendix for more details.
\end{proof}

\begin{remark}[About bias-optimality]
The bound above is bias-optimal for the average function value gap, but not necessarily for the ergodic function value gap $f(\bar{x}_{T}) - \inf f$. This limitation arises from the fact that our Lyapunov analysis is inherently designed to produce bounds on the former quantity, rather than the latter. In the short step-size regime, we benefited from the fortunate circumstance that the worst-case objective is locally linear, which makes the two performance metrics coincide along the iterates. However, this argument does not extend to the present setting, and we are not aware of any existing proof technique capable of establishing SGD bounds directly for the ergodic function value gap.
\end{remark}

\begin{remark}[About the exponential variance]
    The variance term in Theorem \ref{S2::thm:main_large} increases exponentially with $T$, making this bound impractical unless interpolation holds.
    This degenerate variance term is akin to what happened in the critical step-size regime, where we proved that the optimal bias was in general not compatible with a finite variance coefficient.
    We cannot provide a similar proof here, but we show below that one can obtain a reasonable variance term provided the bias term is suboptimal.
\end{remark}

\begin{theorem}[Convex case, large step-sizes, suboptimal bias] \label{S2::th:large subopt}
    Let Assumptions \ref{Ass:strongly convex smooth} and \ref{Ass:bounded solution variance} hold true with $\mu=0$,
    and
    let $(x_t)$ be generated by \eqref{S1::eq:SGD} with $\gamma L \in (1,2)$.
    Then, for every $\varepsilon \in (0, 1)$, every $T \geq 1$, with $\bar x_T = \tfrac{1}{T}\sum_{t=0}^{T-1} x_t$,
    \begin{equation*}
        \Exp{f(\bar x_T) - \min f}
        \leq 
        \frac{1}{T}\sum_{t=0}^{T-1} \mathbb{E} \left[ f(x_t) - \min f \right]
        \leq
        \frac{\Vert x_0 - x_* \Vert^2 }{2\gamma(1-\varepsilon)(2 - \gamma L) T} + 
        \frac{\gamma L + \varepsilon (2- \gamma L)}{2\varepsilon(1-\varepsilon)(2 - \gamma L)^2} \gamma \sigma_*^2.
    \end{equation*}
\end{theorem}

\begin{proof}
    It suffices to select the following parameters in Theorem \ref{T:lyapunov sufficient conditions}:
    \[
    \rho = 2\gamma(1-\varepsilon)(2 - \gamma L), \quad a_t = 1, \quad \alpha_t = \frac{\gamma}{L}, \quad \beta_t = \frac{\gamma}{L}\left[ (\gamma L - 1) + \varepsilon (2 - \gamma L) \right], \quad
    e_t = \gamma^2 \frac{\gamma L + \varepsilon (2- \gamma L)}{\varepsilon(2-\gamma L)}.
    \]
    A detailed proof of their validity, including the motivation for their choice, is presented in Appendix \ref{SA::sec:general}.
\end{proof}

\begin{remark}[Related works]
    As far as we are aware, Theorem \ref{S2::thm:main_large} is the first bound provided for SGD with large step-sizes in the convex smooth case.
    We believe the closest anterior result to be \cite[Theorem 8]{taylor_stochastic_2019} where the authors consider a primal averaging modification of SGD, further assume interpolation to hold, and obtain a convergence rate slightly larger than the bias term of Theorem \ref{S2::thm:main_large}.
    This suggests again that such primal averaging is not needed to derive satisfying bounds for SGD.
    A couple of months after our paper was made available online, a concurrent work independently derived  a bound for SGD in the large step-size regime  \cite[Theorem 8]{attia_fast_2025}, where a suboptimal bias was traded with a bounded variance term.
    This bound corresponds to Theorem \ref{S2::th:large subopt} after taking $\varepsilon = 3/4$ and a variance term slightly larger than ours.
\end{remark}

\section{Bias-Optimal Bounds in the Strongly Convex Smooth Setting}\label{S3::sec:str}\label{S3::sec}

As in the convex case, results for strongly convex problems are obtained after carefully choosing Lyapunov parameters via Theorem \ref{T:lyapunov sufficient conditions}, with a view to derive rates with Lemma \ref{L:SGD bound from lyapunov decrease strongly convex}. 
As in the convex case, the results in this section are presented with slight simplifications, for the ease of presentation. The more detailed bounds may be found in the appendix.
Our first Theorem \ref{S3::thm:strong} provides a new bias-optimal bound which is valid for almost all step-sizes $\gamma L \in (0,2)$.

\begin{theorem}[Strongly convex case, non-critical step-size]\label{S3::thm:strong}
    Let Assumptions \ref{Ass:strongly convex smooth} and \ref{Ass:bounded solution variance} hold, with $\mu > 0$,
and let $(x_t)$ be generated by \eqref{S1::eq:SGD} with $\gamma L \in (0, 2)$. If $\gamma \neq \gamma_{\crit} := \tfrac{2}{\mu + L}$ then, for every $T \geq 1$,
\begin{equation*}
    \Exp{\| x_T-x_*\|^2}\le \phi^{2T}\cdot \|x_0-x_*\|^2 + \frac{1-\phi^{2T}}{1- \phi^2}\frac{\phi \gamma_{\crit}}{\vert \gamma - \gamma_{\crit} \vert} \gamma^2 \sigma_*^2,
\end{equation*}
where $\phi = \max\{ 1- \gamma \mu ; \gamma L -1\} \in [0,1)$. 
Moreover the above bound is bias-optimal, in the sense that there exists a problem for which $\sigma_*^2=0$ and the inequality is an equality. 
\end{theorem}

\begin{proof}
    The proof distinguishes two cases.
    If $L>\mu$, we select the following parameters in Theorem \ref{T:lyapunov sufficient conditions strongly convex}:
\[
	\rho=0, \quad a_t = \phi^{2(T-t)}, \quad \alpha_t = \beta_t= \frac{\gamma \phi}{L-\mu} \cdot a_{t+1}, \quad e_t = \frac{\phi\gamma_{\crit}}{|\gamma-\gamma_{\crit}|} \cdot \gamma^2 a_{t+1}.
\]
    If $L=\mu$, we proceed with a limiting argument.
    We fix $\varepsilon \in (0,1-\phi^2)$ and select the following parameters in Theorem \ref{T:lyapunov sufficient conditions strongly convex}:
    \[
    \rho =0, \quad a_t = (\phi^2+\varepsilon)^{T-t}, \quad \alpha_t = \beta_t = \alpha a_{t+1},\quad e_t = \frac{2 \gamma^2 \alpha}{2 \alpha - \gamma^2}\cdot  a_{t+1}
    \]
    where 
    \[
    	\alpha > \max \left\{ \frac{\gamma^2}{2} \left(1+ \frac{\phi^2}{\varepsilon} \right) ,  \frac{1 - \phi^2- \varepsilon}{2 L^2}, \frac{1}{\varepsilon} \right\}.
    \]
We thus obtain a bound from Lemma \ref{L:SGD bound from lyapunov decrease strongly convex}, and then take $\varepsilon\to 0$.
    A detailed proof of the validity of those parameters, including the motivation for their choice, is presented in Appendix \ref{SA2::sec_str}.
    The bias-optimality of the bound follows from the standard fact that the functions $f(x) = \tfrac{L}{2}\Vert x \Vert^2$ and $\tfrac{\mu}{2}\Vert x \Vert^2$ provide worst-case rates (see Appendix \ref{SA2::sec_str} for the details).
\end{proof}

\begin{remark}[Related works]
	As for the convex case, Theorem \ref{S3::thm:strong} provides a better bound than what can be found in the literature, providing a bias term of the order $\phi^{2T}$, where the geometrical rate $\phi^2$ corresponds exactly to the optimal rate for GD. The result also extends the range of step-sizes to $\gamma L \in (0,2)$, now covering the critical value $\tfrac{2}{\mu + L}$, which is new. 
    The first bounds derived under the same assumptions as Theorem \ref{S3::thm:strong} were obtained in \cite[Theorem 1]{bach_non-asymptotic_2011}. For $\gamma L \in (0, \tfrac{\mu}{L}) \subset (0,1)$, they obtain a bias term $\varphi^T$, where $\varphi = 1 - 2 \mu \gamma + 2L^2 \gamma^2$ is larger than our $\phi^2$, which is $(1-\gamma \mu)^2$ in this context. Those results were improved in \cite[Theorem 2.1]{needell_stochastic_2016}, where the authors extend the analysis to $\gamma L \in (0,1)$ with a bound governed by $\varphi^T$, where $\varphi = 1-2\gamma\mu(1-\gamma L)$ is again larger than our $\phi^2$. A similar bound was obtained more recently in \cite[Theorem 3.1]{gower_sgd_2019}, where the authors notably showed that their result remains valid in a nonconvex setting.
\end{remark}

The above result is valid for every non-critical step-size. 
As $\gamma\to\gamma_{\crit}$, the bias term converges to $\left(\tfrac{L-\mu}{L+\mu} \right)^{2T}$, which is known as the best worst-case rate for GD. 
Unfortunately this is unattainable in Theorem \ref{S3::thm:strong} as the variance term tends to $+\infty$ when $L>\mu$. 
Below, we show that this unattainability is not a flaw of our analysis but rather an intrinsic obstruction, as was the case in the convex setting.

\begin{proposition}[No optimal bias at critical step-size, strongly convex case]\label{thm:sing_str}
    Let $0<\mu<L$ and $T\ge 1$, and let $\gamma_{\crit}=\frac{2}{L+\mu}$ and $\phi_{\crit}=\frac{L-\mu}{L+\mu}$. There exists no finite constant $V$, depending only on $\mu, L$ and $T$, such that 
    \begin{equation}\label{eq:str_to_match}
        \Exp{\|x_T-x_*\|^2}\le \phi_{\crit}^{2T}\|x_0-x_*\|^2+V\sigma_*^2
    \end{equation}
    holds true for every $x_0\in \H$ and every problem satisfying Assumptions \ref{Ass:strongly convex smooth} and \ref{Ass:bounded solution variance} for $\mu>0$, where $(x_t)$ is generated through Algorithm \eqref{S1::eq:SGD} with $\gamma=\gamma_{\crit}$ from $x_0$. 
\end{proposition}
\begin{proof}
    The proof follows by constructing a family of problems whose limit contradicts the finiteness of such a $V$. See details in Appendix \ref{SA2::sec_str}.
\end{proof}

The situation is analogous to the convex case: there is no free-lunch for SGD in the strongly convex case either.
But we prove here as well that an arbitrary small deviation from the optimal bias is enough to restore a  finite variance coefficient.

\begin{theorem}[Strongly convex case, critical step-size, suboptimal bias]\label{S3::thm:strong_opt_bias}
    Let Assumptions \ref{Ass:strongly convex smooth} and \ref{Ass:bounded solution variance} hold true, with $\mu > 0$, and let $\gamma_{\crit} \coloneq \tfrac{2}{\mu + L}$, $\phi_{\crit} = \tfrac{L-\mu}{L+\mu}$.
    Let $(x_t)$ be generated by \eqref{S1::eq:SGD} with $\gamma =\gamma_{\crit}$. 
    Then, for any $\phi \in (\phi_{\crit},1)$ and every $T \geq 1$,
\begin{equation*}
    \Exp{\| x_T-x_*\|^2}\le 
    \phi^{2T}\cdot \|x_0-x_*\|^2 + 
    \frac{1 - \phi^{2T}}{1- \phi^2} \frac{\phi^2}{\phi^2 - \phi_{\crit}^2} \gamma^2 \sigma_*^2.
\end{equation*}
\end{theorem}
\begin{proof}
If $L>\mu$, it suffices to select the following parameters in Theorem \ref{T:lyapunov sufficient conditions strongly convex}:
\[
    \rho =0,\quad a_t = \phi^{2(T-t)}, \quad \alpha_t = \beta_t= \frac{2\phi^2}{(L-\mu)^2}\cdot  a_{t+1}, \quad e_t = \frac{\phi^2}{\phi^2 - \phi_{\crit}^2} \cdot \gamma^2 a_{t+1}.
\]
On the other hand, if $L=\mu$, it suffices to select the following parameters in Theorem \ref{T:lyapunov sufficient conditions strongly convex}:
\[
    \rho =0, \quad a_t = \phi^{2(T-t)}, \quad \alpha_t = \beta_t=\alpha a_{t+1}, \quad e_t =  \frac{2 \gamma^2 \alpha}{2 \alpha - \gamma^2} a_{t+1},
\]
where 
\[
    \alpha > \max \left\{ \frac{\gamma^2}{2} \left(1+ \frac{\phi_{\crit}^2}{\varepsilon} \right) ,  \frac{1 - \phi^2_{\crit} - \varepsilon}{2 L^2}, \frac{1}{\varepsilon} \right\}, \quad \varepsilon = \phi^2 - \phi_{\crit}^2,
\]
and take $\alpha\to +\infty$. A detailed proof of the validity, including the motivation for their choice, is presented in Appendix \ref{SA2::sec_str}.
\end{proof}

As in the convex case, when interpolation holds the variance blowing up is no longer a problem, allowing us to take an optimal bias.

\begin{corollary}[Strongly convex case with interpolation, critical step-size]
Let Assumptions \ref{Ass:strongly convex smooth} and \ref{Ass:bounded solution variance} hold true, with $\mu > 0$,
and assume that $\sigma_*^2=0$. 
Let $(x_t)$ be generated by \eqref{S1::eq:SGD} 
with $\gamma = \tfrac{2}{\mu + L}$. 
Then, for every $T \geq 1$,
\begin{equation*}
    \Exp{\| x_T-x_*\|^2}\le \left( \frac{L-\mu}{L+\mu}  \right)^{2T} \|x_0-x_*\|^2.
\end{equation*}
Moreover the above rate is optimal, in the sense that there exists a problem for which $\sigma_*^2=0$ and the inequality is an equality.
\end{corollary}

\section{Stochastic Proximal Algorithm} \label{SA::sec:prox}

\subsection{Main Result}

In the deterministic setting, one can approximate the minimizers of a nonsmooth (and possibly not everywhere finite) convex function by means of the  Proximal Point Algorithm \citep{martinet_breve_1970,rockafellar_monotone_1976}. In this section, we explain how the study of SGD for the critical step-size $\gamma L = 1$ allows us to study the convergence of the Stochastic Proximal Algorithm as a byproduct, and discuss how the results obtained compare to existing ones.

\begin{problem}\label{Pb:prox}
    Let $\mathcal H$ be a real separable Hilbert space with inner product $\langle \cdot, \cdot \rangle$ and induced norm $\|\cdot\|$.
    Let $\{f_i\}_{i\in \mathcal I}$ be a family of extended real-valued functions $f_i \colon \mathcal H \to \mathbb{R}\cup \{+ \infty\}$, where $\mathcal I$ is a (possibly infinite) set of indices.
    We consider the problem of minimizing $f\coloneqq \Exp{f_i}$, where the expectation is taken 
    with respect to some probability distribution $\mathcal{D}$ over $\mathcal I$, such that $i \mapsto f_i(x)$ is $\mathcal{D}$-measurable.
    We assume that each $f_i$ is proper, convex, and lower semi-continuous.
\end{problem}

An algorithm of choice to solve Problem \ref{Pb:prox} is the \textit{Stochastic Proximal Algorithm} \cite{patrascu_nonasymptotic_2018,asi_stochastic_2019}, defined as
\begin{equation}\label{D:SProx}\tag{SProx}
	x_{t+1} = \prox_{\gamma f_{i_t}}(x_t),
\end{equation}
where $i_t \in \mathcal{I}$ is sampled i.i.d. from the distribution $\mathcal{D}$, and 
where the proximal operator $\prox_{\gamma f_{i}}$ is  defined by
\begin{equation*}
    \prox_{\gamma f_{i}}(x) = \underset{y \in \mathcal{H}}{\rm{argmin}}~\left\{f_i(y) + \frac{1}{2 \gamma} \Vert y-x \Vert^2\right\},
\end{equation*}
with a step-size $\gamma>0$. Since $f_i$ is proper, convex and lower semi-continuous, $\prox_{\gamma f_i}$ is well-defined and single-valued.

The \eqref{D:SProx} algorithm has received significant attention, and benefits from various convergence guarantees. 
One line of work \cite{bianchi_ergodic_2016,combettes_stochastic_2016,combettes_stochastic_nodate,madariaga_convergence_2026}  established the convergence of the iterates for stochastic proximal-type algorithms, often under remarkably weak assumptions. 
Another complementary line of work focuses on establishing explicit complexity rates for \eqref{D:SProx}.
As detailed in the next paragraph, several questions remain open in this direction, and our work allows to  address some of them.

\textbf{Barriers to a Unified Analysis.}  The complexity results available for \eqref{D:SProx} can be divided into four categories: convex feasibility problems; problems where the functions are everywhere finite and uniformly Lipschitz-continuous; smooth problems, where the gradients are uniformly Lipschitz-continuous; and finally, strongly convex problems. These range from functions intrinsically taking the value $+\infty$ to finite functions with various levels of regularity. 
The strongly convex setting is today well understood, and has been studied in several works \cite{richtarik_unified_2024,patrascu_nonasymptotic_2018,asi_stochastic_2019}, with the best known convergence bound established in \cite[Theorem 1]{sadiev2024stochasticproximalpointmethods}.
The convex setting is however more delicate.
As opposed to the deterministic case, there is no obvious performance metric that makes sense in full generality. The first natural candidate is the expected function value gap, with or without averaging, namely $\mathbb{E}\left[f(\bar x_T)-\inf f\right]$ or $\mathbb{E}\left[f(x_T)-\inf f\right]$. However, we cannot hope to bound such a gap for the Stochastic Proximal Algorithm, since nothing guarantees that the iterates $x_T$ (or their averages $\bar x_T$) remain in the domain of $f$. An extreme case is the convex feasibility problem (see Example~\ref{Ex:feasibility problems}), where $\bar x_T$ must already be a solution in order for $f(\bar x_T)$ to be finite. Now, one could search for a bound on the distance between the current iterate and a solution of the problem (or the distance to the set of solutions), but such distance can decrease arbitrarily slowly even in the deterministic case \citep{garrigos_convergence_2023}, which is why we usually do not have bounds on $\Vert x_t-x_*\Vert$ for (non strongly) convex problems. Our analysis will provide bounds for a different metric, which is always well-defined and can be connected to standard performance metrics. 

\textbf{Stochastic Proximal Algorithm as SGD.} 
It is common knowledge that the proximal algorithm is a particular case of the gradient descent algorithm.
This can be stated formally, by introducing the notion of the \textit{Moreau envelope} of a convex function.
Given a proper, convex, and lower semi-continuous function $f\colon \mathcal{H} \to \mathbb{R} \cup \{ + \infty\}$, and a parameter $\gamma > 0$, we define $f^\gamma \colon \mathcal H \to \mathbb{R}$ by
\begin{equation*}
f^\gamma(x) = \inf\limits_{y \in \mathcal H} \left\{ f(y) + \frac{1}{2 \gamma} \Vert y - x \Vert^2\right\}.
\end{equation*}
The Moreau envelope $f^\gamma$ is finite and differentiable. Its gradient, given by $\nabla f^\gamma (x) = \frac{x - \prox_{\gamma f}(x)}{\gamma}$, is Lipschitz continuous with constant $\tfrac{1}{\gamma}$ \citep[Proposition 12.30]{bauschke_convex_2017}. In particular, computing $\prox_{\gamma f}(x)= x - \gamma \nabla f^\gamma(x)$ is the same as computing one step of gradient descent for the function $f^\gamma$, with step-size $\gamma$.
Since \(f^\gamma\) has a $\tfrac{1}{\gamma}$-Lipschitz continuous gradient, Algorithm \eqref{D:SProx} can be seen as an instance of \eqref{S1::eq:SGD} applied to the regularized function
\begin{equation*}
	F^\gamma(x) \coloneqq  \Exp{f_i^\gamma(x)},
\end{equation*}
with the step-size being \emph{critical}.
This connection was already discussed and exploited in \cite{patrascu_nonasymptotic_2018,necoara_randomized_2019}.
The novelty of our approach is that we are able to derive complexity rates for \eqref{D:SProx} from the ones we obtain from  \eqref{S1::eq:SGD} when $\gamma L =1$.
To extend the results established for \eqref{S1::eq:SGD} in the preceding sections to the analysis of \eqref{D:SProx}, we impose the following assumptions on the regularized objective $F^\gamma$, which are analogous to those made in Problem~\ref{Pb:main}.

\begin{assumption} \label{S6:ass:SProx}
    Consider Problem \ref{Pb:prox}. For every $\gamma > 0$, we assume that 
    \begin{enumerate}
    \item the mapping $i \to f_i^\gamma(x)$ is $\mathcal{D}$-measurable and that $\E[f_i^\gamma(x)]$ is finite, for all $x \in \mathcal{H}$;
    \item the mapping $i \to \nabla f^\gamma_i(x)$ is $\mathcal{D}$-measurable and the gradient of $F^\gamma$ verifies $\nabla F^\gamma (x) = \E[\nabla f_i^\gamma(x)]$, for all $x \in \mathcal{H}$;
    \item there exists $x_*^\gamma \in \argmin F^\gamma$.
\end{enumerate}
\end{assumption}

\begin{theorem}[Bound for SProx. General case]\label{T:SProx convex general}
    Let Assumption \ref{S6:ass:SProx} hold.
    Let $(x_t)$ be generated by \eqref{D:SProx} with step-size $\gamma > 0$.
    Then, for every $T\ge 1$, with $\bar x_T=\tfrac1T\sum_{t=0}^{T-1}x_t$,
    \begin{equation*}
    \mathbb{E}\left[  F^\gamma( \bar x_T) - \min F^\gamma \right] \leq 
    \frac{\Vert x_0 - x_*^\gamma \Vert^2}{\gamma T} + 3 \gamma \sigma^2_*(\gamma),
    \end{equation*}
    where $\sigma^2_*(\gamma) = \Exp{\Vert \nabla f_i^\gamma(x_*^\gamma) \Vert^2 }$ is assumed finite, and $x_*^\gamma \in {\rm{argmin}}~F^\gamma$.
\end{theorem}

\begin{proof}
    Apply Theorem \ref{S2::thm:main_1} with $\varepsilon = 1/2$. 
\end{proof}

To the best of our knowledge, this is the first complexity bound for the Stochastic Proximal Algorithm under mere convexity, without additional regularity assumptions such as Lipschitz continuity or smoothness. In contrast, the analyses in \cite{asi_stochastic_2019,davis_stochastic_2019,patrascu_nonasymptotic_2018} rely on Lipschitz continuity.
This comes at the cost of assuming that $\argmin F^\gamma \neq \emptyset$, which is not guaranteed under the sole requirement of $\argmin f \neq \emptyset$.
Further, we note some unusual terms in those bounds, such as the use of the regularized function $F^\gamma=\Exp{f_i^\gamma}$ for the metric instead of $f=\Exp{f_i}$, or the use of the regularized minimizer $x_*^\gamma$ instead of $x_*$.
The presence of $x_*^\gamma$ is an artifact of our approach, that we are unable to avoid, except for interpolated problems. We leave this technical issue for a future work.
In the following subsections we illustrate that the optimality metric $F^\gamma$ can be related to metrics used in the literature under standard additional assumptions, allowing us to recover several classical results, together with a novel result in the interpolation setting.

\subsection{Interpolated Problems}

Let us remind, as stated in Definition \ref{D:interpolation}, that interpolation holds when the functions $f_i$ share a common minimizer.
Using standard convex analysis tools, one can show that if the family $(f_i)_{i \in \mathcal{I}}$ interpolates, then so does $(f_i^\gamma)_{i \in \mathcal{I}}$, which means that we can use Corollary \ref{S2::coro:inter_1}.
We further have ${\rm{argmin}}~F^\gamma = {\rm{argmin}}~f$ and $\inf F^\gamma = \inf f$, which means that the function value gap $F^\gamma(x) - \inf F^\gamma$ is a meaningful metric for our original problem, in the sense that it provides a faithful  measure of how far we are from optimality.
This leads to the next result, which  is new in the convex interpolated case.

\begin{theorem}[Bounds for SProx. Interpolation case]\label{T:Sprox convex interpolation}
    Let Assumption \ref{S6:ass:SProx} hold, and suppose that interpolation holds.
    Let $(x_t)$ be generated by \eqref{D:SProx} with step-size $\gamma > 0$.
	Then, for every $T\ge 1$, with $\bar x_T=\tfrac1T\sum_{t=0}^{T-1}x_t$ and $x_* \in {\rm{argmin}}~f$,
	\begin{equation*}
	\mathbb{E}\left[  F^\gamma( \bar x_T) - \inf f \right] \leq 
	\frac{\Vert x_0 - x_* \Vert^2}{2 \gamma T}.
	\end{equation*}
\end{theorem}

\begin{proof}
    Let us first verify that interpolation for the original problem also applies to the regularized one.
    If interpolation holds for the family $(f_i)_{i \in \mathcal{I}}$, then 
    $\cap~ {\rm{argmin}}~f_i \neq \emptyset$.
    As shown, for instance, in \cite[Proposition 12.9]{bauschke_convex_2017}, $\inf f_i^\gamma = \inf f_i$ and ${\rm{argmin}}~f_i^\gamma = {\rm{argmin}}~f_i$.
    As such $\cap~ {\rm{argmin}}~f_i^\gamma \neq \emptyset$ and interpolation holds for the regularized problem.
    We can then use  Lemma \ref{L:interpolation argmin and inf coincide} to see that
    \[
        {\rm{argmin}}~F^\gamma 
        = \bigcap {\rm{argmin}}~f_i^\gamma 
        = \bigcap{\rm{argmin}}~f_i
        =
        {\rm{argmin}}~f\quad\hbox{and}\quad \inf F^\gamma = \Exp{\inf f_i^\gamma}
        = \Exp{\inf f_i}
        = \inf f.
    \]
	We can now apply Corollary \ref{S2::coro:inter_1}.
\end{proof}

\begin{example}[Feasibility problems]\label{Ex:feasibility problems}
    Feasibility problems consist in finding $x_* \in C \coloneqq \cap_i C_i$ where $C_i$ are nonempty closed convex sets.
    A standard algorithm for solving this problem is the \textit{Stochastic Projection Algorithm}, where the current iterate is projected onto a randomly chosen set $C_i$: $x_{t+1} = \proj_{C_{i_t}}(x_t)$.
This is a particular case of the \eqref{D:SProx} algorithm applied to Problem \ref{Pb:prox} where $f_i$ is taken as the indicator function of $C_i$, which is $0$ if $x \in C_i$ and is $+ \infty$ if $x \notin C_i$.
A first important observation is that this problem is well-posed ($f$ has minimizers) if and only if interpolation holds.
This is due to the fact that ${\rm{argmin}}~f=C \coloneqq \cap_i C_i=\cap_i{\rm{argmin}}~f_i$.
A second observation is that the regularized function value gap is simply $F^\gamma(x) - \inf F^\gamma
		=
		\tfrac{1}{2 \gamma}\Exp{\dist(x;C_i)^2}$
which is a standard measure of how far $x$ is from the set of solutions $C = \cap_i C_i$.
We can therefore apply Theorem \ref{T:Sprox convex interpolation}  and obtain that the Stochastic Projection Algorithm enjoys the rate 
	\begin{equation*}
	\Exp{\dist(\bar x_T;C_i)^2}
    \leq 
	\frac{\dist(x_0; C)^2}{ T}.
	\end{equation*}
The above bound is well-known and corresponds exactly to what was obtained in \citep[Proposition 6]{nedic_random_2010}.
It is remarkable that our analysis on SGD  converted into \eqref{D:SProx} is able to provide exactly the best known rates for the Stochastic Projection Algorithm.
Of course the story does not stop here, as it is possible to derive better bounds provided that the intersection $\cap_i C_i$ is \emph{regular} \cite{Kru06,LewLukMal09}.
Such regularity ensures that the problem enjoys a Łojasiewicz inequality \cite{AttBolSva13,bolte_error_2017} which can be exploited to derive linear rates of convergence  \cite[Theorem 3]{necoara_randomized_2019}.
We leave for future work the question whether our analysis can be adapted to exploit such regularity features.
\end{example}

\subsection{Regular Problems: Lipschitz, Smooth}

If the functions $f_i$ are  everywhere finite, the function value gap $f(x) - \inf f$ becomes an appropriate performance metric. It is therefore reasonable to expect the results of  Theorem \ref{T:SProx convex general} to hold for $f(x) - \inf f$ instead of $F^\gamma(x) - \inf F^\gamma$.
We are going to exploit the fact that both quantities are related when the problem is \emph{regular}, which we split in two settings: Lipschitz and smooth problems.

\begin{corollary}[Bound for SProx. Regular cases]\label{T:SProx lipschitz:cor}
    Consider Problem \ref{Pb:prox}, and suppose that Assumption \ref{S6:ass:SProx} holds.
	Let $(x_t)$ be generated by \eqref{D:SProx}, with step-size $\gamma >0$.
    \begin{enumerate}
        \item If each function $f_i$ is $G$-Lipschitz continuous, then
	\begin{equation*}
	\mathbb{E}\left[  f( \bar x_T) - \inf f \right] \leq 
	\frac{\Vert x_0 - x_*^\gamma \Vert^2}{\gamma T} 
    + 4\gamma G^2, \qquad\hbox{where}\quad x_*^\gamma \in {\rm{argmin}}~F^\gamma.
	\end{equation*}
        \item If each function $f_i$ is $L$-smooth, if $\mathbb{E}[\inf f_i] > - \infty$, and if $\gamma L \in (0,1)$, then 
	\begin{equation*}
	\mathbb{E}\left[  f( \bar x_T) - \inf f \right] \leq 
	\frac{\Vert x_0 - x_*^\gamma \Vert^2}{(1- \gamma L)\gamma T} 
    +   
    \frac{7 \gamma L}{1 - \gamma L} \Delta_*, \qquad\hbox{where}\quad x_*^\gamma \in {\rm{argmin}}~F^\gamma,\quad \Delta_* = \inf f - \Exp{\inf f_i}.
    \end{equation*}
    \end{enumerate}
\end{corollary}

\begin{proof}
In this proof we will use the subdifferential of $f_i$ at a point $x$, which is the set
\begin{equation*}
    \partial f_i(x) \coloneqq \{ g \in \mathcal{H} \ | \ (\forall y \in \mathcal{H}) \ f_i(y) - f_i(x) - \langle g, y-x \rangle \geq 0 \}.
\end{equation*}
This set is nonempty, convex and closed  \citep[Propositions 16.4 and 16.27]{bauschke_convex_2017}, in particular due to the fact that we assume $f_i$ to take finite values. We denote its least-norm element by $\partial^0 f_i(x)$.
Using the definition of subdifferential, we can write for $x \in \mathcal{H}$
\begin{equation*}
    f_i(x) - f_i^\gamma(x) = \sup_{y \in \mathcal{H}} f_i(x) - f_i(y) - \frac{1}{2 \gamma} \Vert y-x \Vert^2 
        \leq  \sup_{y \in \mathcal{H}} - \langle \partial^0 f_i(x), y-x \rangle - \frac{1}{2\gamma} \Vert y - x \Vert^2 
\end{equation*}
The supremum on the right-hand side is attained at $y = x -\gamma \partial^0 f_i(x)$, which yields
$f_i(x) - f_i^\gamma(x) 
        \leq 
        \frac{\gamma}{2} {\Vert \partial^0 f_i(x) \Vert^2}$. 
We further know that $f_i^\gamma \leq f_i$, which means that $F^\gamma \leq f$ and so $\inf F^\gamma \leq \inf f$.
Combining all the above inequalities, we deduce that
    \begin{equation}\label{e:transfer regularized function value}
        (\forall x \in \mathcal{H}) \quad
        f(x) - \inf f \leq F^\gamma(x) - \inf F^\gamma + \frac{\gamma}{2} \Exp{\Vert \partial^0 f_i(x) \Vert^2}.
    \end{equation}
Combining this inequality with Theorem \ref{T:SProx convex general}, we obtain
	\begin{equation*}
	\mathbb{E}\left[  f( \bar x_T) - \inf f \right] \leq 
	\frac{\Vert x_0 - x_*^\gamma \Vert^2}{\gamma T} 
    + 3 \gamma \sigma_*^2(\gamma) 
    + \frac{\gamma}{2} \Exp{\Vert \partial^0 f_i(\bar x_T) \Vert^2},
	\end{equation*}
	where $x_*^\gamma \in {\rm{argmin}}~F^\gamma$ and $\sigma^2_*(\gamma) = \mathbb{E}\left[\Vert \nabla f_i^\gamma(x_*^\gamma) \Vert^2 \right]$.
We are now ready to derive our bounds in the Lipschitz and smooth cases.
\begin{enumerate}
    \item The $G$-Lipschitzness of $f_i$ guarantees that $\Vert \partial^0 f_i(\bar x_T) \Vert^2 \leq G^2$.
    Moreover, $\|\nabla f_i^\gamma (x_*^\gamma)\| \le \|\partial^0 f_i(x_*^\gamma)\|$ (see \citep[Lemma 3]{patrascu_nonasymptotic_2018}) so we also have $\Vert \nabla f_i^\gamma(x_*^\gamma) \Vert^2 \leq G^2$ and in turn $\sigma_*^2(\gamma) \leq G^2$.
    This leads to the desired variance term $3 \gamma G^2 + \tfrac{\gamma}{2}G^2 \leq 4 \gamma G^2$.
    \item If the functions are $L$-smooth, it is easy to bound $\sigma_*^2(\gamma)$ with $\Delta_*$.
    To do so, 
    use the convexity and smoothness of $f_i$ through \cite[Lemma 2.28]{garrigos_handbook_2024} to get 
    \begin{equation*}
        \Vert \nabla f_i^\gamma(x_*^\gamma) \Vert^2
        \leq
        2L \left( f_i^\gamma(x_*^\gamma) -  \inf f_i^\gamma \right)
        =
        2L \left( f_i^\gamma(x_*^\gamma) -   \inf f_i \right).
    \end{equation*}
    After taking expectation ($i \to \inf f_i$ is measurable, see Proposition \ref{SD::prop:inf measurable}), we conclude that $\sigma_*^2(\gamma) 
        \leq 
        2L \left( \inf F^\gamma - \Exp{\inf f_i} \right)
        \leq 
        2L \left( \inf f  - \Exp{\inf f_i}  \right)
        =
        2L \Delta_*$.
    Then we need to tackle the term $\Exp{\Vert \partial^0 f_i(\bar x_T) \Vert^2} = \Exp{\Vert \nabla f_i(\bar x_T) \Vert^2}$.
    A standard variance transfer argument (see \citep[Lemma 4.19]{garrigos_handbook_2024}) shows that $\Exp{\Vert \nabla f_i(x) \Vert^2} \leq 2L(f(x) - \inf f) + 2L \Delta_*$, where $\Delta_* \coloneqq  \inf f - \Exp{\inf f_i}$.
    Reorganizing the term $f(\bar x_T) - \inf f$ and dividing by $(1-  \gamma L)$ leads to the desired complexity bound.
\end{enumerate}
\end{proof}

\begin{remark}[Convex Regular Problems] \label{R: convex regular}
    The unifying approach that leads to Theorem \ref{T:SProx convex general} comes at a price.
    Directly exploiting the regularity of the functions allows to derive sharper bounds.
    For Lipschitz functions, both \cite[Proposition 4]{asi_stochastic_2019} and \cite[Theorem 5]{patrascu_nonasymptotic_2018} obtain bounds in which the bias and variance terms are both improved by constant multiplicative factors, namely 
    \[
    \E \left[f(\overline{x}_T) - \inf f\right] \leq \frac{\|x_0 - x_*\|^2}{2\gamma T} + \frac{\gamma G^2}{2}.
    \]
    Likewise, for smooth functions \cite[Theorem 4.3]{traore_variance_2024} obtains the better bound 
    \[
    \E \left[f(\overline{x}_T) - \inf f\right] \leq \frac{\|x_0-x_*\|^2}{\gamma T} + 2\gamma \sigma^2_*,
    \]
    for a step-size $\gamma < 1/2L$. Additionally, note that the right hand side of the aforementioned bounds involve a true problem solution $x_*$, whereas ours unconventionally relies on the distance of the initialization $x_0$ to the regularized solution $x_*^\gamma$.
\end{remark}

\section{Conclusion and Perspectives}\label{S4::sec}\label{SC::sec}

In this paper, we established new and improved upper bounds for \eqref{S1::eq:SGD} without assuming uniformly bounded gradient variance, both in the convex and strongly convex settings, and for a full range of step-sizes. 
Our analysis relies on the introduction of a novel Lyapunov energy \eqref{S0::eq:Lyapunov}. 
By providing lower bounds that match our results, we demonstrated that the obtained bounds are \textit{bias-optimal}.
One exception remains, at critical step-sizes, which provide the best worst-case rates for GD. 
We showed in this regime that no bias-optimal bounds can be obtained in general, but that suboptimal bounds are possible. 

\textbf{Future Work. }\label{SC::sec:future} 
We leave the following questions and research directions open:

\begin{itemize}
    \item While this paper focuses on achieving optimal bias terms, this is not the only meaningful objective. In future work, we plan to investigate bounds that provide sharp \emph{complexity} rates. We anticipate this to be more challenging, as quantities such as $\|x_0 - x_*\|^2$ and $\sigma_*^2$ then become intrinsic problem parameters rather than fixed constants.
    \item In this work, we focused on Lyapunov energies of the form \eqref{S0::eq:Lyapunov}. It remains unclear whether sharper results could be obtained by considering more general classes of Lyapunov functions. While we showed that the bias term usually cannot be improved, no such guarantees are currently available for the variance term. If they are not optimal, this might be due to our chosen Lyapunov structure.
    \item In the convex setting, our results establish bias-optimality for the average function value gap in the large step-size regime. A natural next step would be to strengthen Theorem \ref{S2::thm:main_large} in order to obtain bias-optimal bounds for the function value gap evaluated at the Ces\`aro average of the iterates.
    This goes way beyond the scope of this paper, given that most proof techniques we are aware of rely on the average function value gap.
    \item Our analysis is restricted to constant step-sizes. Extending these results to varying or adaptive step-size schemes constitutes an interesting and potentially impactful line of work.
\end{itemize}

\textbf{Acknowledgments. } 
This work benefited from the support of the FMJH Program Gaspard Monge for optimization and operations research and their interactions with data science. 
Daniel Cortild acknowledges the support of the Clarendon Funds Scholarships.
The authors also thank the Center for Information Technology of the University of Groningen for their support and for providing access to the Hábrók high performance computing cluster.

\appendix


\section{Proofs of Main Results}\label{SA::sec:proofs}

\subsection{Convex Setting}\label{SA::sec:conv}\label{SA::sec:conv_proof}

In this section, we derive parameters that satisfy the sufficient conditions of Theorem \ref{T:lyapunov sufficient conditions strongly convex} in order to derive upper bounds for SGD in the convex setting. 
Instead of simply verifying symbolically that our proposed parameters verify Theorem \ref{T:lyapunov sufficient conditions strongly convex},  we try to provide intuition on how these parameters were derived.

\subsubsection{Short Step-Sizes: \texorpdfstring{$\gamma L <1$ (Proof of Theorem \ref{S2::thm:main_small})}{}}\label{SA::sec:proof_short}

\begin{proposition}[Lyapunov parameters. Convex case, short step-sizes]\label{P:lyapunov convex short}
    Let $\gamma L \in (0, 1)$, and consider the parameters $\rho, a_t, e_t, \alpha_t, \beta_t$ defined by
    \begin{itemize}
        \item $\rho = 2 \gamma + \tfrac{2(1 - \gamma L)}{LT}$,
        \item $a_t = \tfrac{T-t}{T} \tfrac{1+ \gamma L (T-1) }{1 + \gamma L(T-t-1)}$,
        \item $\alpha_t \equiv \alpha =\tfrac{\rho}{2L}$,
        \item $\beta_t \equiv \beta =0$,
        \item $e_t = \frac{a_{t+1}\gamma^2 \alpha_t }{\alpha_t - a_{t+1} \gamma^2}$.
    \end{itemize}
    These parameters satisfy the sufficient Lyapunov conditions of Theorem \ref{T:lyapunov sufficient conditions strongly convex}.
    Moreover, it holds that
    \begin{equation*}
        e_t = 
        \frac{\gamma^2}{1 - \gamma L} \frac{T-t-1}{T} \frac{(1 - \gamma L) + \gamma L T}{(1- \gamma L) + \gamma L(T-t)} 
        \leq
        \frac{\gamma^2}{1 - \gamma L}.
    \end{equation*}
    In particular, the bound in Theorem \ref{S2::thm:main_small} holds true.
\end{proposition}

\begin{proof}
    Notice that the sufficient conditions from Theorem \ref{T:lyapunov sufficient conditions} are homogeneous. Because our goal is to obtain bounds whose bias term is the smallest possible, we can without loss of generality impose that $a_0=1$, and try to maximize $\rho$.
    To simplify the search of parameters, we impose $\alpha_t \equiv \alpha >0$ and $\beta_t\equiv \beta=0$.
    We dedicate the remainder of the proof to justify our remaining choices.
    
    We start by focusing on Condition \ref{T:lyapunov sufficient conditions strongly convex:curious} from Theorem \ref{T:lyapunov sufficient conditions}, which reads
    \begin{eqnarray*}
        && \left( a_{t+1} \gamma - L{\alpha} \right)^2 
        \leq   
        ({\alpha} - a_{t+1} \gamma^2) (a_t - a_{t+1}) \\
        & \iff & 
        \left( a_{t+1} \gamma - L{\alpha} \right)^2 
        \leq   
        (1-\gamma L) \alpha(a_t - a_{t+1}) +(\gamma L \alpha - a_{t+1} \gamma^2) (a_t - a_{t+1}) 
        \\
        & \iff & 
        \left( a_{t+1} \gamma - L{\alpha} \right)^2 
        - \gamma ( L \alpha - a_{t+1} \gamma) (a_t - a_{t+1}) 
        \leq   
        (1-\gamma L) \alpha(a_t - a_{t+1})
        \\
        & \iff & 
        \left(L \alpha - a_{t+1} \gamma \right)(L \alpha - a_{t+1} \gamma
        - \gamma  (a_t - a_{t+1}) )
        \leq   
        (1-\gamma L) \alpha(a_t - a_{t+1}) 
        \\
        & \iff & 
        \left(L \alpha - a_{t+1} \gamma \right)(L \alpha
        -  a_t \gamma)
        \leq   
        (1-\gamma L) \alpha(a_t - a_{t+1}).
    \end{eqnarray*}
    Condition \ref{T:lyapunov sufficient conditions strongly convex:monotone at} boils down to $a_t\ge a_{t+1}$. By additionally imposing $a_t>a_{t+1}$, we may divide by $a_t-a_{t+1}$ to obtain
    \begin{eqnarray*}
        & \iff & 
        \frac{ 
        \left(L \alpha - a_{t+1} \gamma \right)(L \alpha 
        - a_t \gamma )}{a_t - a_{t+1}}
        \leq   
        (1-\gamma L) \alpha
        \\
        & \iff & 
        \frac{ 
        \left(L \alpha - a_{t+1} \gamma \right)(L \alpha 
        - a_t \gamma )}{(L \alpha - a_{t+1} \gamma) - (L \alpha - a_{t} \gamma ) }
        \leq   
        \frac{(1-\gamma L) \alpha}{\gamma}.
    \end{eqnarray*}
    As $a_t$ is decreasing, with $a_0=1$, it holds true that $a_t\le 1$ for all $t=0, \ldots, T-1$. Specifically, if we assume that $L\alpha>\gamma$, then both $L \alpha - a_{t+1} \gamma>0$ and $L \alpha - a_{t} \gamma>0$ hold true, thus allowing us to write
   \begin{eqnarray*}
        & \iff & 
        \frac{(L \alpha - a_{t+1} \gamma) - (L \alpha - a_{t} \gamma ) }{ 
        \left(L \alpha - a_{t+1} \gamma \right)(L \alpha
        - a_t \gamma )}
        \geq   
        \frac{\gamma}{(1-\gamma L) \alpha}
        \\
        & \iff & 
        \frac{1}{L \alpha 
        - a_t \gamma}
        -
        \frac{1}{L \alpha - a_{t+1} \gamma}
        \geq   
        \frac{\gamma}{(1-\gamma L) \alpha}
        \\
        & \iff & 
        \frac{1}{L \alpha - a_{t+1} \gamma}
        \leq   
        \frac{1}{L \alpha 
        - a_t \gamma}
        -
        \frac{\gamma}{(1-\gamma L) \alpha}.
    \end{eqnarray*}
    Introducing the temporary variables $u_t \coloneqq  \tfrac{1}{L \alpha - \gamma  a_t}$ and $c = \tfrac{\gamma}{(1-\gamma L) \alpha}$, we see that we obtained an arithmetic inequality $u_{t+1} \leq u_t - c$, leading us to $u_t \leq u_0 - ct$.
    We know that $u_0 = \tfrac{1}{L \alpha - \gamma}$, so we can write
    \begin{equation*}
        u_t \leq 
        \frac{1}{L \alpha - \gamma} - \frac{t\gamma}{(1 - \gamma L) \alpha}
        =
        \frac{(1 - \gamma L) \alpha - t\gamma (L \alpha - \gamma)}{(1- \gamma L) \alpha (L \alpha - \gamma)}.
    \end{equation*}
    Returning to the definition of $u_t$ we deduce the following bound for $a_t$:
    \begin{eqnarray*}
        a_t &\leq &
        \frac{1}{\gamma}\left( L \alpha - \frac{(1- \gamma L) \alpha (L \alpha - \gamma)}{(1 - \gamma L) \alpha - t \gamma (L \alpha - \gamma)} \right) \\
        &=&
        \frac{1}{\gamma}\frac{ L \alpha^2(1- \gamma L) -L\alpha t \gamma  (L \alpha - \gamma) - (1 - \gamma L) \alpha (L \alpha - \gamma)}{(1 - \gamma L) \alpha - t \gamma (L \alpha - \gamma)} 
        \\
        &=&
        \frac{1}{\gamma} \frac{  -L \alpha t  \gamma (L \alpha - \gamma)  +   \alpha\gamma(1 - \gamma L)}{(1 - \gamma L) \alpha - t \gamma (L \alpha - \gamma)} \\
        &=&
        \alpha \frac{  (1 - \gamma L) -t  L (L \alpha - \gamma)}{(1 - \gamma L) \alpha - t \gamma (L \alpha - \gamma)}.
    \end{eqnarray*}
    We note that our Condition \ref{T:lyapunov sufficient conditions strongly convex:positivity} requires that $a_t \geq 0$ for all $t=0, \ldots, T$, which is equivalent to having $a_T \geq 0$ as we assumed $a_t$ to be decreasing.
    Specifically, we must have
    \begin{equation*}
         (1 - \gamma L) - T L (L \alpha - \gamma) \geq 0 
         \quad
         \iff
         \quad
         \alpha \leq \frac{1 - \gamma L + \gamma T L}{T L^2}.
    \end{equation*}
    As we want $\rho$ as large as possible, and due to Condition \ref{T:lyapunov sufficient conditions strongly convex:rho} which reads $\rho \leq 2L(\alpha - \beta)$ with $\beta = 0$, we want $\alpha$ to be as large as possible.
    Therefore we fix 
    \begin{equation*}
        \alpha = \frac{1 - \gamma L + T \gamma L}{T L^2}
        \quad \text{ and } \quad 
        \rho = 2L \alpha = \frac{2(1 - \gamma L) + 2 T \gamma L}{T L}.
    \end{equation*}
    This choice respects the assumption $L \alpha > \gamma$ we made earlier, since $1 - \gamma L > 0$.
    
    To get an expression for $a_t$, we simply set the inequalities to equalities, and replace $\alpha$ by its chosen value.
    We then obtain
    \begin{equation*}
        L \alpha - \gamma=
        \frac{(1- \gamma L) + \gamma LT - \gamma LT}{TL}
        =
        \frac{1 - \gamma L}{TL}
        \Rightarrow
        u_0 = \frac{TL}{1 - \gamma L}.
    \end{equation*}
    So
    \begin{eqnarray*}
        u_t &=& u_0 - tc 
        = \frac{TL}{1 - \gamma L} - \frac{t \gamma}{\alpha(1- \gamma L)}
        =
        \frac{\alpha TL - t \gamma }{\alpha(1- \gamma L)} 
        =
        \frac{\tfrac{1}{L}(1 - \gamma L) + \gamma T - t \gamma }{\alpha(1- \gamma L)}     
        =
        \gamma \frac{T -t - 1 + \tfrac{1}{\gamma L}}{\alpha (1 - \gamma L)}.
    \end{eqnarray*}
    We now use the fact that $\gamma a_t = L \alpha - u_t^{-1}$ to write
    \begin{eqnarray*}
        a_t 
        &=&
        \frac{L\alpha}{\gamma}  - \frac{1}{\gamma u_t}
        =
        \frac{\alpha L}{\gamma } - \frac{1}{L \gamma} \frac{(1- \gamma L) \tfrac{\alpha L}{\gamma}}{T-t-1 + \tfrac{1}{\gamma L}} \\
        &=&
        \frac{\alpha L}{\gamma}\left( 1 - \frac{1}{\gamma L} \frac{(1- \gamma L) }{T-t-1 + \tfrac{1}{\gamma L}} \right)
        =
        \frac{\alpha L}{\gamma}\left( 1 - \frac{(1- \gamma L) }{\gamma L (T-t-1 )+ 1} \right) \\
        &=&
        \frac{\alpha L}{\gamma} \frac{\gamma L (T-t)}{\gamma L (T-t-1 )+ 1} 
        =
        \frac{(1 - \gamma L) + \gamma L T}{T \gamma L} \frac{\gamma L (T-t)}{\gamma L (T-t-1 )+ 1} \\
        &=&
        \frac{(T-t)}{T} \frac{1  + \gamma L (T-1)}{\gamma L (T-t-1 )+ 1}.
    \end{eqnarray*}
    To conclude the proof, we need to check all the sufficient Lyapunov conditions of Theorem \ref{T:lyapunov sufficient conditions}.
    Conditions \ref{T:lyapunov sufficient conditions strongly convex:positivity}, \ref{T:lyapunov sufficient conditions strongly convex:rho}, \ref{T:lyapunov sufficient conditions strongly convex:monotone at} and \ref{T:lyapunov sufficient conditions strongly convex:curious} are already satisfied, by construction.
    It remains to check Conditions \ref{T:lyapunov sufficient conditions strongly convex:annoying} and \ref{T:lyapunov sufficient conditions strongly convex:variance}.

    Condition \ref{T:lyapunov sufficient conditions strongly convex:annoying} is equivalent to
    \begin{eqnarray*}
        a_{t+1} \gamma^2 \leq \alpha
        &\iff&
         \gamma^2 \frac{(T-t- 1)}{T} \frac{1  + \gamma L (T-1)}{1 + \gamma L (T-t-2 )} \leq \frac{(1 - \gamma L) + \gamma LT}{T L^2} \\
         & \iff &
         \frac{\gamma^2 L^2 (T-t-1)}{1 + \gamma L (T-t-2 )} (1 + \gamma L (T-1)) \leq 1 + \gamma L (T-1)\\
         & \iff &
         \gamma^2 L^2 (T-t-1) \leq 1 + \gamma L (T-t-2 )
         =1 + \gamma L (T-t-1 ) - \gamma L
         \\
         & \iff &
         0\leq 1 - \gamma L + \gamma L (T-t-1 )(1 - \gamma L),
    \end{eqnarray*}
    which is true for every $t = 0, \dots, T-1$.

    Condition \ref{T:lyapunov sufficient conditions strongly convex:variance} is a lower bound on $e_t$, of which the equality case coincides with the chosen value of $e_t$. More specifically, this gives
    \begin{eqnarray*}
       e_t = \frac{a_{t+1}\gamma^2 \alpha }{\alpha - a_{t+1} \gamma^2}
       =
       \frac{\gamma^2}{\frac{1}{a_{t+1} } - \frac{\gamma^2}{\alpha}},
    \end{eqnarray*}
    where
    \begin{eqnarray*}
        \frac{1}{a_{t+1}} - \frac{\gamma^2}{\alpha}
        &=&
        \frac{T}{T-t - 1} \frac{(1- \gamma L) + \gamma L (T-t-1) }{(1 - \gamma L) + \gamma L T}
        -
        \frac{T \gamma^2 L^2}{(1 - \gamma L) + \gamma L T}
        \\
        &=&
        T \frac{(1- \gamma L) + \gamma L (T-t-1) - \gamma^2 L ^2 (T-t-1)}{(T-t-1)((1 - \gamma L) + \gamma L T)} \\
        &=&
        T \frac{(1- \gamma L) ((1- \gamma L) + \gamma L(T-t))}{(T-t-1)((1 - \gamma L) + \gamma L T)}.
    \end{eqnarray*}
    We thus get
    \begin{eqnarray*}
        e_t &=& \frac{\gamma^2}{1 - \gamma L} \frac{T-t-1}{T} \frac{(1 - \gamma L) + \gamma L T}{(1- \gamma L) + \gamma L(T-t)}.
    \end{eqnarray*}
    We readily verify that $e_t$ is nonincreasing with respect to $t$.
    Indeed, its derivative with respect to $t$ has the same sign as
    \begin{equation*}
        (T-1)((1 - \gamma L) + \gamma L T) \gamma L - ((1 - \gamma L) + \gamma L T)^2,
    \end{equation*}
    which is always nonpositive.
    Therefore $e_t \leq e_0$, where
    \begin{equation*}
        e_0 = \frac{\gamma^2}{1 - \gamma L} \frac{T-1}{T} \leq \frac{\gamma^2}{1 - \gamma L},
    \end{equation*}
    thus concluding the proof. The bound in Theorem \ref{S2::thm:main_small} follows by combining Theorem \ref{T:lyapunov sufficient conditions} and Lemma \ref{L:SGD bound from lyapunov decrease}.
\end{proof}

Finally, we prove the bias-optimality claim in Theorem \ref{S2::thm:main_small}.

\begin{proposition}[Bias-optimality. Convex case, short step-sizes]\label{SC::prop:GD}
        The bias obtained in the bound of Theorem \ref{S2::thm:main_small} is optimal, in the sense that, for any $x_0\in \mathcal H\backslash \{0\}$, there exists a convex and smooth problem for which $\mathcal I=\{1\}$ and both inequalities are satisfied with equality. Moreover, this problem has bounded gradients.
\end{proposition}

\begin{proof}
    We build our lower bound example similarly to the construction in \cite{taylor_smooth_2017}.
    The idea is to consider Huber functions, which have minimizers while behaving linearly when far away from the minimizers.
    This linear behavior is not only slow, which fits our goal to find a worst-case function, but also allows for computing easily and explicitly the iterates.
    Given $\eta >0$, we define the  function $\mathcal{H}_{\eta}$ as
\begin{align*}
    \mathcal{H}_{\eta}(x) = \begin{cases}
        \eta L \|x\| - \frac{L}{2}\eta^2 &\text{if} \quad \|x\| > \eta \\
        \frac{L}{2}\|x\|^2 &\text{if} \quad \|x\| \leq \eta.
    \end{cases}
\end{align*}
This is exactly the classical Huber function multiplied by $L$, which is convex, $L$-smooth and has bounded gradients \cite[Example 13.7]{bauschke_convex_2017}. 
It now remains to prove the claim.
As a preliminary, observe that if $\Vert x \Vert \geq \eta$ and if $x^+ := x - \gamma \nabla \mathcal{H}_\eta(x)$, then
\begin{equation*}
    x^+ = x(1 - \tfrac{\gamma L \eta}{\Vert x \Vert}), 
    \text{ with } \Vert x^+ \Vert = \Vert x \Vert - \gamma L \eta
    \text{ and } 
    \tfrac{x^+}{\Vert x^+ \Vert} =  \tfrac{x}{\Vert x \Vert}.
\end{equation*}
Let $x_0 \in \mathcal H$ and denote $D=\Vert x_0 \Vert$, and let $(x_t)$ be the iterates of gradient descent applied to $f = \mathcal{H}_{\eta}$, where the parameter $\eta$ is defined as $\eta = \tfrac{D}{1+(T-1)\gamma L}$.
A simple induction argument shows that this choice of $\eta$ is small enough to ensure that $\Vert x_t \Vert \geq \eta$ for $t = 0, \dots, T-1$, which in turn allows us to write
$
    x_t =  (1 - t \tfrac{\gamma L \eta}{D})x_0,
$
and as such
\begin{align*}
    \overline{x}_{T} &= \frac{1}{T} \sum_{t = 0}^{T-1} x_t
    = \frac{1}{T} \sum_{t = 0}^{T-1} \left(1 - t\gamma L\eta /D\right)x_0
    = \left(1 - \frac{\gamma L(T-1)\eta}{2D} \right) x_0.
\end{align*}
A direct computation shows that $\Vert \bar x_T \Vert \ge \eta$, which means that
\begin{align*}
    f(\overline{x}_{T}) - \inf f &= \eta L \left(D - \frac{\gamma L (T-1)\eta}{2}\right) - \frac{L}{2} \eta^2
    =\eta LD -L\eta^2\cdot \frac{1+ \gamma L (T-1)}{2} 
    = \eta LD-\frac{\eta LD}{2} = \frac{\eta LD}{2} = \frac{L \|x_0-x_*\|^2}{2 (1+(T-1)\gamma L)},
\end{align*}
which indeed matches the bias term of Theorem \ref{S2::thm:main_small}. Moreover, the iterates $x_0, \dots, x_{T-1}$ all belong to the segment $[x_0,x_{T-1}]$, on which $f$ is affine, which means that we also have $\frac{1}{T} \sum_{t=0}^{T-1} f(x_t) - \inf f = f(\overline{x}_{T}) - \inf f$, thus concluding.
\end{proof}

\subsubsection{General Step-Size, Suboptimal Bias (Proof of Theorems \ref{S2::thm:main_1} and \ref{S2::th:large subopt})}\label{SA::sec:general}

We now provide a proof for our suboptimal bounds, which cover at the same time the critical and large step-size regimes.

\begin{theorem}[Convex case, any step-size, suboptimal bounds]\label{T:convex large stepsize subopt:appendix}
Let $\gamma L \in (0, 2)$ and $\varepsilon\in (0, 1)$, and consider the parameters $\rho, a_t, e_t, \alpha_t, \beta_t$ defined by
    \begin{itemize}
        \item $\rho = 2\gamma(1-\varepsilon)\eta$ where $\eta = 1 - (\gamma L - 1)_+$,
        \item $a_t = 1$,
        \item $\alpha_t = \gamma / L$,
        \item $\beta_t = \alpha_t - \tfrac{\rho}{2L}$,
        \item $e_t=\gamma^2 \frac{\alpha_t + \beta_t}{\alpha_t + \beta_t - \gamma^2}$.
    \end{itemize}
    These parameters satisfy the sufficient Lyapunov conditions of Theorem \ref{T:lyapunov sufficient conditions}. Moreover, the following simplified bound
    \begin{equation*}
    \frac{\bar{e}_T}{\rho} \leq \gamma \cdot \frac{\gamma L +\varepsilon \eta}{2(1-\varepsilon)\varepsilon \eta^2}
    \end{equation*}
    holds true. In particular, Theorems \ref{S2::thm:main_1} and \ref{S2::th:large subopt} hold true.
\end{theorem}
\begin{proof}
    The goal is to find parameters satisfying the sufficient Lyapunov conditions of Theorem \ref{T:lyapunov sufficient conditions} with a suboptimal value for $\rho$.
    Therefore we start by defining $\rho = 2\gamma(1-\varepsilon)\eta$ for some $\varepsilon \in (0,1)$ and where $\eta = 1$ or $2-\gamma L$ depending whether $\gamma L \leq 1$ or $\gamma L\ge 1$.
    We take a  sequence $a_t$ constantly equal to $1$, which automatically satisfies Condition \ref{T:Lyap:3}.
    By Condition \ref{T:Lyap:5}, we must have $\alpha_t = \gamma / L$.
	We note that Conditions \ref{T:Lyap:2} and \ref{T:Lyap:4} respectively provide upper and lower bounds on $\beta_t$.
    In view of satisfying Condition \ref{T:Lyap:6} with the smallest possible value for $e_t$, we aim to pick $\beta_t$ to be as large as possible.
    We therefore saturate the upper bound on $\beta_t$ in Condition \ref{T:Lyap:2} and take 
    \begin{equation*}
        \beta_t = \alpha_t - \frac{\rho}{2L}
        =
        \frac{\gamma}{L} - \frac{2\gamma(1-\varepsilon)\eta}{2L}
        =
        \frac{\gamma}{L} \left[ 1 - \eta + \eta  \varepsilon \right].
    \end{equation*}
    This choice is compatible with Condition \ref{T:Lyap:4} as
    \begin{equation*}
        \gamma^2 \leq \alpha_t + \beta_t
        \Leftrightarrow 
        \gamma L - 1 \leq 1 - \eta + \eta  \varepsilon,
    \end{equation*}
	which is true since
    \begin{equation*}
    \gamma L - 1 \leq (\gamma L - 1)_+ = 1 - \eta \leq 1 - \eta + \eta \varepsilon.
    \end{equation*}
    Finally, the choice of $e_t$ is such that Condition \ref{T:Lyap:6} is an equality, and hence also satisfied. We note that Condition \ref{T:Lyap:1} is verified, as $\beta_t\ge 0$ since $\rho\le 2\gamma$, and the other parameters are naturally nonnegative.

    To derive the simplified bound it suffices to compute
    \begin{equation*}
        \frac{\bar{e}_T}{\rho} = \frac{\gamma^2}{\rho} \frac{2 - \eta(1- \varepsilon)}{2- \gamma L - \eta(1- \varepsilon)}
        =
        \frac{\gamma}{2(1 - \varepsilon)\eta} \frac{2 - \eta(1- \varepsilon)}{2- \gamma L - \eta + \eta \varepsilon}.
    \end{equation*}
    We may simplify the second fraction of this variance term, by considering two cases;
    \begin{itemize}
    	\item if $\gamma L \geq 1$, then $\eta = 2 - \gamma L$ and
    	\begin{equation*}
    	\frac{2 - \eta(1- \varepsilon)}{2- \gamma L - \eta + \eta \varepsilon}
    	=
    	\frac{2 - (2-\gamma L)(1- \varepsilon)}{\eta \varepsilon}
    	=
    	\frac{\gamma L + (2-\gamma L) \varepsilon}{\eta \varepsilon}
    	=
    	\frac{\gamma L + \eta \varepsilon}{\eta \varepsilon}.
    	\end{equation*}
    	\item if $\gamma L \leq 1$, then $\eta = 1$ and
    	\begin{equation*}
    	\frac{2 - \eta(1- \varepsilon)}{2- \gamma L - \eta + \eta \varepsilon}
    	=
    	\frac{1+\varepsilon}{1 - \gamma L + \varepsilon}.
    	\end{equation*}
    	The above fraction can be seen as a convex function $\phi(\gamma L)$ of $\gamma L \in (0,1)$.
    	As such we can write that $\phi( \gamma L) \leq \phi(0) + \gamma L (\phi(1) - \phi(0))$
    	where $\phi(0) = 1$ and $\phi(1) = \tfrac{1+\varepsilon}{ \varepsilon}$,
    	which in our case implies that
    	\begin{equation*}
    	\frac{1+\varepsilon}{1 - \gamma L + \varepsilon}
    	\leq
    	\frac{\gamma L +\varepsilon}{\varepsilon}
    	=
    	\frac{\gamma L +\varepsilon \eta}{\varepsilon \eta}.
    	\end{equation*}
    \end{itemize}
    In both cases the second fraction of our variance term is bounded by $\tfrac{\gamma L +\varepsilon \eta}{\varepsilon \eta}$, which gives the desired bound.
    
    Theorems \ref{S2::thm:main_1} and \ref{S2::th:large subopt} now follow by combining Theorem \ref{T:lyapunov sufficient conditions} and Lemma \ref{L:SGD bound from lyapunov decrease}, and noting that $\eta=1$ and $\eta=2-\gamma L$ when $\gamma L=1$ and $\gamma L\ge 1$, respectively.
\end{proof}

\subsubsection{Critical Step-Size: \texorpdfstring{$\gamma L =1$}{}}\label{SA3::ssec:optimal}\label{SA::sec:proof_opt}

We start by showing that the optimal bias cannot be attained at the critical step-size.

\begin{proof}[of Proposition \ref{thm:sing_cvx}]
    Let  $m \geq 1$ be an integer, let $\H=\R$ and consider $m$ functions sampled uniformly:
    \[
        f_1(x)=\cdots =f_{m-1}(x)=\frac{L}{2}x^2-x\quad \text{and}\quad f_m(x)=\frac{p}{1-p}x,
    \]
    where $p=1-\frac{1}{m}$.
    All those functions are convex and $L$-smooth, and one can compute that
    \[
        f(x)=\frac{pL}{2}x^2, \quad \min f = 0,  \quad x_*=0, \quad \text{and}\quad \sigma_*^2=\frac{p}{1-p}.
    \]
    Note that
    \[
        x_1=\begin{cases}
            \frac{1}{L} & \text{with probability $p$} \\
            x_0-\frac{p}{L(1-p)} & \text{with probability $1-p$}.
        \end{cases}
    \]
    In particular, 
    \[
        \Exp{x_1^2}=\frac{p}{L^2}+(1-p)\left(x_0-\frac{p}{L(1-p)}\right)^2=\frac{p}{L^2}+(1-p)x_0^2+\frac{p^2}{L^2(1-p)}-\frac{2p}{L}x_0.
    \]
    Now it holds that 
    \begin{align*}
        \frac{1}{T}\sum_{t=0}^{T-1}\Exp{f(x_t)-\min f}
        &\ge \frac{1}{T}\sum_{t=0}^{1}\Exp{f(x_t)-\min f} \\
        &=\frac{pL}{2T}(x_0^2+\Exp{x_1^2})\\
        &=\frac{pL}{2T}\left(x_0^2+\frac{p}{L^2}+(1-p)x_0^2+\frac{p^2}{L^2(1-p)}-\frac{2p}{L}x_0\right).
    \end{align*}
    If a bias-optimal bound was possible, we would have
    \[
        V\frac{p}{1-p}=V\sigma_*^2\ge \frac{1}{T}\sum_{t=0}^{T-1}\Exp{f(x_t)-\min f}-\frac{L\|x_0-x_*\|^2}{2T}\ge \frac{L}{2T}\left(px_0^2+\frac{p^2}{L^2}+p(1-p)x_0^2+\frac{p^3}{L^2(1-p)}-\frac{2p^2}{L}x_0-x_0^2\right),
    \]
    or equivalently,
    \begin{align*}
        V
        &\ge \frac{L}{2T}\left((1-p)x_0^2+\frac{p(1-p)}{L^2}+(1-p)^2x_0^2+\frac{p^2}{L^2}-\frac{2p(1-p)}{L}x_0-\frac{(1-p)x_0^2}{p}\right)\\
        &= \frac{L}{2T}\left(\frac{-(1-p)^3}{p}x_0^2+\frac{p}{L^2}-\frac{2p(1-p)}{L}x_0\right).
    \end{align*}
    Taking $x_0=\frac{-p^2}{L(1-p)^2}$, which maximizes the above expression, we obtain
    \[
        V\ge \frac{L}{2T}\left(-\frac{p^3}{L^2(1-p)}+\frac{p}{L^2}+\frac{2p^3}{L^2(1-p)}\right)=\frac{1}{2LT}\left(\frac{p^3}{1-p}+p\right).
    \]
    Taking $p\to 1$, or equivalently $m\to \infty$, shows that no such $V$ can exist, proving the claim.
\end{proof}

\subsubsection{Large Step-Sizes: \texorpdfstring{$\gamma L > 1$ (Proof of Theorem \ref{S2::thm:main_large})}{}}\label{SA::sec:proof_large}

\begin{proposition}[Lyapunov parameters. Convex case, large step-sizes]\label{P:lyapunov convex large step-size}
    Let $1< \gamma L <2$, and consider the parameters $\rho, a_t, e_t, \alpha_t, \beta_t$ defined by
    \begin{itemize}
        \item $\rho = \tfrac{2 \gamma(2 - \gamma L)}{1 - \theta^T}$ where $\theta = ( \gamma L - 1)^2$,
        \item $a_t = \tfrac{1 - \theta^{T-t}}{1- \theta^T}$,
        \item $\beta_t = \tfrac{ \gamma(\gamma L -1)(1 - \theta^{T-t-1})}{L(1-\theta^T)}$,
        \item $\alpha_t = \gamma \tfrac{ (2 - \gamma L) + (\gamma L -1)(1 - \theta^{T-t-1})}{L(1-\theta^T)}$,
        \item $e_t = 
        \tfrac{\gamma^2}{2-\gamma L} 
        \tfrac{1 - \theta^{T-t-1}}{1- \theta^T} 
        \left( \tfrac{\gamma L }{\theta^{T-t-1}} - 2(\gamma L -1) \right)$.
    \end{itemize}
    These parameters satisfy the sufficient Lyapunov conditions of Theorem \ref{T:lyapunov sufficient conditions}.
    Moreover, the averaged variance $\bar e_T \coloneqq  \tfrac{1}{T} \sum_{t=0}^{T-1} e_t$ has the following asymptotic equivalence as $T \to + \infty$:
    \begin{equation*}
        \bar e_T \sim 
        \frac{\gamma^2}{(2 - \gamma L)^2 } \frac{1}{T (1 - \gamma L)^{2T-2}} \to + \infty.
    \end{equation*}
    In particular, the bound in Theorem \ref{S2::thm:main_large} holds true.
\end{proposition}

\begin{proof}
    We begin by imposing
    \begin{equation*}
        \beta_t = \frac{a_{t+1}\gamma (\gamma L -1)}{L},
    \end{equation*}
    and saturating Condition \ref{T:lyapunov sufficient conditions:rho upper bound}, meaning that
    \begin{equation*}
        \rho = 2L(\alpha_t - \beta_t) \iff
        \alpha_t = \frac{\rho }{2L} + \beta_t.
    \end{equation*}
    Considering Condition \ref{T:lyapunov sufficient conditions:curious}, we have
    \begin{equation*}
        (a_{t+1} \gamma - L \alpha_t)^2 \le (\alpha_t + \beta_t - \gamma^2 a_{t+1})(a_t - a_{t+1}).
    \end{equation*}
    The latter will give us an expression of $a_t$ in terms of $a_{t+1}$, from which we will deduce a general expression for $a_t$ by induction.
    By injecting the values of $\alpha_t$ and $\beta_t$, we obtain
    \begin{equation*}
        \frac{1}{L^2}\left( \frac{\rho L}{2} - a_{t+1} \gamma L(2-\gamma L) \right)^2 \le
        \frac{1}{L^2}
        \left( \frac{\rho L}{2} - a_{t+1} \gamma L(2-\gamma L)  \right) (a_t - a_{t+1}).
    \end{equation*}
    If $\frac{\rho L}{2} - a_{t+1} \gamma L(2-\gamma L) > 0$, which holds whenever $\rho > 2\gamma (2 - \gamma L)$ (since Condition \ref{T:Lyap:3} requires $(a_t)$ to be decreasing, and we may impose without loss of generality $a_0 = 1$), the latter is equivalent to
    \begin{eqnarray*}
        \frac{\rho L}{2} - a_{t+1} \gamma L(2-\gamma L) \le a_t - a_{t+1}
        \iff
        a_t \ge \frac{\rho L}{2} + a_{t+1} (\gamma L -1)^2.
    \end{eqnarray*}
    We observe an arithmetic-geometric relation which can be summarized after setting $\theta = (\gamma L -1)^2$:
    \begin{equation*}
        a_t \ge \frac{\rho L}{2} + a_{t+1} \theta \iff a_{t+1} \le \theta^{-1} a_t - \frac{\rho L}{2 \theta}.
    \end{equation*}
    A simple induction argument leads to
    \begin{equation*}
        a_t \le \theta^{-t} a_0 - \frac{\rho L}{2 \theta} \frac{\theta^{-t} - 1}{\theta^{-1} - 1}.
    \end{equation*}
    As done previously, without loss of generality, we impose $a_0 = 1$.
    We impose the sequence to be nonincreasing and nonnegative. The latter leads to
    \begin{eqnarray*}
        a_T \geq 0 
         \iff 
        \frac{\rho L}{2 \theta} \frac{\theta^{-T} - 1}{\theta^{-1} - 1} \leq \theta^{-T} 
         \iff 
        \frac{\rho L}{2} \leq 
        \theta^{-T} \frac{1 - \theta}{\theta^{-T} -1}
        =
        \frac{1-\theta}{1 - \theta^T}.
    \end{eqnarray*}
    As we aim to maximize $\rho$, we impose
    \begin{equation*}
        \frac{\rho L}{2} 
        =
        \frac{1-\theta}{1 - \theta^T}
        \iff
        \rho = \frac{2(1-\theta)}{L(1 - \theta^T)}= \frac{2\gamma (2 - \gamma L)}{1 - \theta^T}.
    \end{equation*}
    In particular, we obtain
    \begin{equation*}
        a_t \le \theta^{-t} - \frac{\rho L}{2} \frac{\theta^{-t} - 1}{1 - \theta}
        =
        \theta^{-t} - \frac{1-\theta}{1 - \theta^T}\frac{\theta^{-t} - 1}{1 - \theta}
        =
        \theta^{-t} - \frac{\theta^{-t} - 1}{1 - \theta^T}
        =
        \frac{1 - \theta^{T-t}}{1 - \theta^T}.
    \end{equation*}
    To enforce that $(a_t)$ is nonincreasing, we set the above to be an equality,
    and since $a_T\ge 0$, it is also nonnegative, as wanted.

    We now verify the remaining conditions from Theorem \ref{T:lyapunov sufficient conditions}.
    The nonnegativity of $a_t$ implies the nonnegativity of $\alpha_t, \beta_t$.
    Conditions \ref{T:lyapunov sufficient conditions:positivity}, \ref{T:lyapunov sufficient conditions:monotone at}, \ref{T:lyapunov sufficient conditions:curious} are verified by construction.
    It remains to check that Conditions \ref{T:lyapunov sufficient conditions:annoying} and \ref{T:lyapunov sufficient conditions:variance} hold true.
    
    In the following, we introduce $k \coloneqq T-t-1$ to simplify notation.
    To check that Condition \ref{T:lyapunov sufficient conditions:annoying} is satisfied, we compute $\alpha_t + \beta_t$:
    \begin{equation*}
        \alpha_t + \beta_t = \frac{\rho }{2L} + 2 \beta_t
        =
        \frac{\gamma (2 - \gamma L) + 2 \gamma (\gamma L - 1)(1- \theta^k)}{L(1 - \theta^T)}.
    \end{equation*}
    Therefore, Condition \ref{T:lyapunov sufficient conditions:annoying}, which states $\alpha_t + \beta_t - \gamma^2 a_{t+1} \geq 0$, becomes
    \begin{eqnarray*}
        0 &\leq & \frac{\gamma (2 - \gamma L) + 2 \gamma (\gamma L - 1)(1- \theta^k)}{L(1 - \theta^T)} - \gamma^2  \frac{1 - \theta^k}{1 - \theta^T} \\
        &=&
        \frac{\gamma (2 - \gamma L) + 2 \gamma (\gamma L - 1)(1- \theta^k) - \gamma^2 L (1 - \theta^k)}{L(1 - \theta^T)} \\
        &=&
        \frac{\gamma (2 - \gamma L) - (1 - \theta^k) \gamma (2 - \gamma L)}{L(1 - \theta^T)} \\
        &=&
        \frac{\gamma (2 - \gamma L)  \theta^k}{L(1 - \theta^T)},
    \end{eqnarray*}
    which is true because $\gamma L \in (1,2)$.
    It remains to study Condition \ref{T:lyapunov sufficient conditions:variance}, which is satisfied by setting 
    \begin{eqnarray*}
        e_t &=& \gamma^2
        \frac{a_{t+1}(\alpha_t + \beta_t)}{\alpha_t + \beta_t - \gamma^2 a_{t+1}} \\
        &=&
        \gamma^2 
        \frac{1 - \theta^k}{1- \theta^T} 
        \frac{\gamma (2 - \gamma L) + 2 \gamma (\gamma L - 1)(1- \theta^k)}{L(1 - \theta^T)}
        \frac{L(1 - \theta^T)}{\gamma (2 - \gamma L)  \theta^k}\\
        &=&
        \gamma^2 
        \frac{1 - \theta^k}{1- \theta^T} 
        \frac{ (2 - \gamma L) + 2  (\gamma L - 1)(1- \theta^k)}{ (2 - \gamma L)  \theta^k} \\
        &=&
        \gamma^2 
        \frac{1 - \theta^k}{1- \theta^T} 
        \frac{ \gamma L - 2  (\gamma L - 1) \theta^k}{ (2 - \gamma L)  \theta^k} \\
        &=&
        \frac{\gamma^2 }{2 - \gamma L}
        \frac{1 - \theta^k}{1- \theta^T} \left( \frac{ \gamma L}{\theta^k} - 2  (\gamma L - 1) \right).
    \end{eqnarray*}
    To conclude the proof, it remains to analyze the average of the $e_t$'s.
    Denote $e_t'' = ({1 - \theta^k}) \left( \frac{ \gamma L}{\theta^k} - 2  (\gamma L - 1) \right)$, such that
    \begin{eqnarray*}
        \sum_{t=0}^{T-1} e_t''
        &=&
        \gamma L 
        \sum_{t=0}^{T-1} \frac{1 - \theta^k}{\theta^k}
        - 2(\gamma L -1) \sum_{t=0}^{T-1} (1 - \theta^k) \\
        &=&
        \gamma L 
        \sum_{t=0}^{T-1} (\theta^{-k} - 1)
        - 2(\gamma L -1) \sum_{t=0}^{T-1} (1 - \theta^k) \\
        &=&
        \gamma L 
        \left( \sum_{t=0}^{T-1} \theta^{-k} \right) - T \gamma L
        - 2(\gamma L -1) T + 2(\gamma L -1)\left(  \sum_{t=0}^{T-1}  \theta^k \right) \\
        &=&
        \gamma L 
        \left( \sum_{t=0}^{T-1} \theta^{-t} \right) - T(3 \gamma L - 2) + 2(\gamma L -1)\left(  \sum_{t=0}^{T-1}  \theta^t \right) \\
        &=&
        \gamma L 
        \frac{1 - \theta^{-T}}{1 - \theta^{-1}} - T(3 \gamma L - 2) + 2(\gamma L -1) \frac{1 - \theta^T}{1 - \theta} \\
        &=&
        \frac{\theta}{2 - \gamma L} 
        (\theta^{-T}-1) - T(3 \gamma L - 2) + \frac{2(\gamma L -1)}{\gamma L (2 - \gamma L)} (1 - \theta^T) \\
        &=&
        \frac{1}{(2 - \gamma L)\theta^{T-1}} 
        - \frac{\theta}{2 - \gamma L}
        - T(3 \gamma L - 2) + \frac{2(\gamma L -1)}{\gamma L (2 - \gamma L)} (1 - \theta^T).
    \end{eqnarray*}
    First, observe that the above quantity is asymptotically equivalent to $\tfrac{\theta}{2-\gamma L} \theta^{-T}$, because the exponentially growing term $\theta^{1-T}$ dominates the other ones.
    Second, we can use the fact that $T \geq 1$ to bound the dominated terms
    \begin{eqnarray*}
        - \frac{\theta}{2 - \gamma L}
        - T(3 \gamma L - 2) + \frac{2(\gamma L -1)}{\gamma L (2 - \gamma L)} (1 - \theta^T)
        & \leq &
        - \frac{(\gamma L -1)^2}{2 - \gamma L}
        - (2 \gamma L - 2) + \frac{2(\gamma L -1)}{\gamma L (2 - \gamma L)} \\
        & = &
        \frac{2(\gamma L -1) - \gamma L(\gamma L -1)^2}{\gamma L (2 - \gamma L)}
        - 2(\gamma L - 1)\\
        & = &
        (\gamma L -1) \left[ \frac{2 - \gamma L (\gamma L -1) - 2 \gamma L (2 - \gamma L)}{\gamma L (2 - \gamma L)}  \right] \\
        & = &
        \frac{(\gamma L -1)}{\gamma L (2 - \gamma L)} \left[  {2 - \gamma L(\gamma L -1) - 2 \gamma L (2 - \gamma L)} \right] \\
        & = &
        \frac{(\gamma L -1)}{\gamma L (2 - \gamma L)} \left[  {2 + (\gamma L)^2 - 3 \gamma L} \right] \\
        & = &
        \frac{(\gamma L -1)}{\gamma L (2 - \gamma L)} \left[  (\gamma L - 1)(\gamma L - 2)\right] \\
        & = &
        \frac{-(\gamma L -1)^2}{\gamma L } \leq 0.
    \end{eqnarray*}
    From the above we deduce that for all $T \geq 1$
    \begin{equation*}
        \sum_{t=0}^{T-1} e_t'' \leq \frac{1}{(2 - \gamma L)\theta^{T-1}}.
    \end{equation*}
    Therefore we can write
    \begin{equation*}
        \bar e_T
        =
        \frac{\gamma^2 }{2 - \gamma L}
        \frac{1}{1- \theta^T} 
        \frac{1}{T} \sum_{t=0}^{T-1} e_t''
        \leq 
        \frac{\gamma^2 }{(2 - \gamma L)^2}
        \frac{1}{1- \theta^T} 
        \frac{1}{T \theta^{T-1}}.
    \end{equation*}
    
    \if{where $e_t' = (2 - \gamma L) e_t''$ and
    \begin{equation}\label{eq:variance lyapunov convex large step-size}
        \bar e_T'
        =
        \frac{1}{T}\sum_{t=0}^{T-1} e_t'
        =
        \frac{\theta}{T}  
        (\theta^{-T}-1) - (2-\gamma L)(3 \gamma L - 2) + \frac{2(\gamma L -1)}{\gamma L T} (1 - \theta^T).
    \end{equation}
    In particular, one sees that $\bar e_T' \sim \frac{1}{T \theta^{T-1}}$, where we recall that we defined $\theta = (\gamma L -1)^2 \in (0,1)$.}\fi 
    The bound in Theorem \ref{S2::thm:main_large} now follows by combining Theorem \ref{T:lyapunov sufficient conditions} and Lemma \ref{L:SGD bound from lyapunov decrease}, after observing that the variance term can be simplified:
    \begin{equation*}
        \frac{\bar e_T \sigma_*^2}{\rho}
        \leq 
        \frac{\gamma \sigma^2_* }{2(2 - \gamma L)^3}
        \frac{1}{T \theta^{T-1}}
    \end{equation*}
\end{proof}

We may moreover show the bias-optimality claim of Theorem \ref{S2::thm:main_large} by providing a function that achieves the lower bound.

\begin{proposition}[Bias-optimality. Convex case, large step-sizes]\label{SD::sec:proof_optimality_rates}
    The bias obtained in the bound of Theorem \ref{S2::thm:main_large} is optimal, in the sense that, for any $x_0\in \mathcal H$, there exists a convex and smooth problem for which $\mathcal{I} = \{1\}$ and the inequality is satisfied with equality.
\end{proposition}
\begin{proof}
    Consider the function $f(x) = \tfrac{L}{2}\|x\|^2$ with minimal value $\inf f=0$ and minimizer $x_*=\mathbf 0$, and take any $x_0 \in \mathcal H$. The iterates $(x_t)$ of gradient descent are then given by $x_t = (1-\gamma L)^t x_0$, such that
    \begin{align*}
        \frac{1}{T}\sum_{t=0}^{T-1} f(x_t) - \inf f &= \frac{L}{2T}\sum_{t=0}^{T-1} (1-\gamma L)^{2t}\cdot \|x_0\|^2 = \frac{1 - (1-\gamma L)^{2T}}{2\gamma(2-\gamma L)T}\cdot \|x_0-x_*\|^2.
    \end{align*}
    This indeed matches the bias term in Theorem \ref{S2::thm:main_large}.
\end{proof}

We finally conclude with a consequence of all our results for the convex case, which is a sharp characterization of Gradient Descent's rates for the average function value metric.

\begin{corollary}[Optimal average function value rates for GD]\label{C:GD optimal rates}
    Let $f : \mathcal{H} \to \mathbb{R}$ be a convex and $L$-smooth function, and let $x_* \in {\rm{argmin}}~f \neq \emptyset$.
    Consider $(x_t)$ the sequence generated by the Gradient Descent algorithm with stepsize $\gamma > 0$ starting from $x_0 \in \mathcal{H}$.
    If $\gamma L \in (0,2)$ then for every $T \geq 1$
    \begin{equation*}
        \frac{1}{T}\sum_{t=0}^{T-1} f(x_t) - \inf f
        \leq 
        \Vert x_0 - x_*\Vert^2
        \begin{cases}
            \frac{L}{2 \gamma L T + 2 (1-\gamma L)} & \text{ if } \gamma L \in (0,1], \\
            \frac{(1-(\gamma L -1)^{2T})L}{2 \gamma (2 - \gamma L)T} & \text{ if } \gamma L \in [1,2).
        \end{cases}
    \end{equation*}
    Those rates are optimal, in the sense that for every choice of $\gamma \in (0, \tfrac{2}{L})$, there exists a convex and $L$-smooth function for which the above inequality is an equality.
    In particular, the best worst-case rate for Gradient Descent with constant stepsize, relatively to the average function value metric, is attained when taking $\gamma L = 1$, which gives
    \begin{equation*}
        \frac{1}{T}\sum_{t=0}^{T-1} f(x_t) - \inf f
        \leq 
            \frac{L \Vert x_0 - x_*\Vert^2}{2T}.
    \end{equation*}
\end{corollary}

\begin{proof}
    The upper bound follows from our bounds for SGD that we obtained in Theorem \ref{S2::thm:main_small}, Corollary \ref{S2::coro:inter_1}, and Theorem \ref{S2::thm:main_large}.
    The optimality for the bound for $\gamma L \neq 1$ was also proven in Theorems \ref{S2::thm:main_small} and \ref{S2::thm:main_large}.
    The optimality of the bound for $\gamma L =1$ follows from the same function that we used for $\gamma L <1$, see the proof of Theorem \ref{S2::thm:main_small}.
    The fact that $\gamma L = 1$ is the best worst-case choice follows from the fact that
    \begin{equation*}
        \lambda \in (0,1] \mapsto \frac{L}{2 \lambda T + 2 (1-\lambda)} 
        \text{ is nonincreasing, and } 
        \lambda \in [1,2) \mapsto \frac{(1-(\lambda -1)^{2T})L}{2 \lambda (2 - \lambda)T} \text{ is nondecreasing}.
    \end{equation*}
\end{proof}

\subsection{Strongly Convex Setting}\label{SA2::sec}\label{SA2::sec_str}

As in the previous section, we derive parameters that satisfy the sufficient conditions of Theorem \ref{T:lyapunov sufficient conditions strongly convex} in order to derive upper bounds for SGD, this time in the strongly convex setting. We again note that all propositions could simply be verified symbolically, yet we include the proofs to provide intuition on how these parameters were derived.

As in the convex case, our upper bounds become singular when considering the critical step-size.
We therefore provide a relaxed bound, where either the bias term is non-optimal ($\varepsilon>0$), or the step-size is non-critical ($\gamma \neq \gamma_{\crit}$).

\begin{proposition}[Lyapunov parameters. Strongly convex case, any step-size, $L > \mu$]\label{P:lyapunov strongly convex any step-size}
    Let $ \gamma L \in (0,2)$, with $L > \mu$.
    Let  $\gamma_{\crit} = \tfrac{2}{\mu + L}$, $\phi_\opt = \max\{ 1- \gamma \mu ; \gamma L -1\}$, and $0 \leq \varepsilon < 1 - \phi_\opt^2$. Assume that either $\varepsilon >0$ or $\gamma \neq \gamma_{\crit}$.
    Consider the parameters $\rho, a_t, e_t, \alpha_t, \beta_t$ defined by
    \begin{itemize}
        \item $\rho =0$,
        \item $a_t = \phi^{2(T-t)}$, where $\phi^2 = \phi_{\opt}^2 + \varepsilon$,
        \item $\alpha_t = \beta_t= \alpha a_{t+1}$, where $\alpha = \omega + \omega_\varepsilon$ and
    \begin{equation*}
        \omega = \frac{\gamma \phi_{\opt}}{L-\mu} \quad \text{ and } \quad 
        \omega_\varepsilon =
        \frac{\varepsilon + \sqrt{\varepsilon^2 +  \varepsilon \gamma (L  -  \mu) (L + \mu) \delta }}{(L-\mu)^2}
        \quad \text{ with } \quad 
        \delta = \vert \gamma - \gamma_{\crit} \vert,
    \end{equation*}
        \item $e_t = e a_{t+1}$, where
        \begin{equation*}
            e = \gamma^2 \left( 1 + \frac{\gamma^2(L- \mu)}{2 \omega_\varepsilon (L- \mu) 
        + ({L+ \mu})
         \gamma \delta} \right).
        \end{equation*}
    \end{itemize}
    Then the sufficient Lyapunov conditions of Theorem \ref{T:lyapunov sufficient conditions strongly convex} are verified.
    Moreover, we have
    \begin{equation*}
        e_T^{sum} \coloneqq \sum_{t=0}^{T-1} e_t = e \frac{1-\phi^{2T}}{1- \phi^2} \leq \frac{e}{1 - \phi^2},
    \end{equation*}
    and 
    \begin{equation*}
        e \leq \gamma^2 \left( 1 + \frac{\gamma^2 (L - \mu)^2}{4 \varepsilon + (L+ \mu)(L- \mu) \gamma \delta} \right).
    \end{equation*}
\end{proposition}

\begin{proof}
    Since we seek a distance-bound, we set $\rho=0$. We further consider $\alpha_t = \beta_t$ and ${a_t}/{a_{t+1}}=\phi^2\in (0, 1)$.
    Since $\phi_{\opt}^2$ is the optimal rate for the deterministic gradient descent in the strongly convex smooth case, we anticipate that $\phi^2\ge \phi_\opt^2$.
    Without loss of generality, we impose $a_T=1$, from which we deduce that $a_t = \phi^{2(T-t)}$.
    We moreover impose that $e_t = e a_{t+1}$ and that $\alpha_t = \alpha a_{t+1}$ for some constants $e, \alpha > 0$.    
    
    It remains to study the conditions of Theorem \ref{T:lyapunov sufficient conditions strongly convex} to find feasible values for $\alpha$ and  $e$.

    Condition \ref{T:lyapunov sufficient conditions strongly convex:variance} yields
    \begin{equation*}
        a_{t+1} \gamma^2 (2\alpha a_{t+1} ) \leq (2\alpha a_{t+1} - a_{t+1}\gamma^2)e a_{t+1},
    \end{equation*}
    which holds true if
    \begin{equation*}
         2\gamma^2 \alpha \leq (2\alpha  - \gamma^2)e.
    \end{equation*}
    As we want to minimize $e$, we set
    \begin{equation*}
        e = \frac{2\gamma^2 \alpha}{2\alpha  - \gamma^2},
    \end{equation*}
    and impose that the denominator $2\alpha  - \gamma^2$ is positive.
    Now it remains to find $\alpha$. The above expression is decreasing with $\alpha$, so we aim to find the largest $\alpha$ possible such that all conditions are satisfied.

    Turning to Condition \ref{T:lyapunov sufficient conditions strongly convex:curious} will determine $\alpha$.
    With our current notations and restrictions on the parameters, it can be written as 
    \begin{equation*}
        (a_{t+1}\gamma - \alpha a_{t+1} (L + \mu) )^2 \leq (2\mu L \alpha a_{t+1}  + a_t - a_{t+1})(2 \alpha a_{t+1}  - a_{t+1}\gamma^2),
    \end{equation*}
    which is true if
    \begin{equation*}
        (\gamma - \alpha (L + \mu) )^2 - (2\mu L \alpha  -(1- \phi^2))(2 \alpha   - \gamma^2) \leq 0.
    \end{equation*}
    This is a polynomial of degree at most $2$ in $\alpha$, which must be nonpositive.
    Expanding the terms to find the coefficients of this polynomial $P(\alpha)$, we obtain
    \begin{eqnarray*}
        P(\alpha) 
        &=&
        \gamma^2 + \alpha^2(L+\mu)^2 - 2 \alpha (L+ \mu)\gamma 
        - 4\mu L \alpha^2 + 2 \mu L \alpha \gamma^2 + 2 \alpha (1 - \phi^2) - \gamma^2(1 - \phi^2) \\
        &=&
        \alpha^2 \left[ (L+\mu)^2 - 4 \mu L \right]
        - \alpha \left[ 2 \gamma(L + \mu) - 2 \mu L \gamma^2 - 2(1- \phi^2) \right]
        + \left[ \gamma^2 - \gamma^2(1 - \phi^2) \right] \\
        &=&
        a \alpha^2 - 2b \alpha + c,
    \end{eqnarray*}
    where
    \begin{equation}\label{D:sgd strong convex curious polynomial coefficients}
        \begin{cases}
        a &= (L+\mu)^2 - 4 \mu L = (L-\mu)^2, \\
        b &= \gamma(L + \mu) -  \mu L \gamma^2 - (1- \phi^2), \\
        c &= \gamma^2 \phi^2 = \gamma^2 ( \phi_{\opt}^2 + \varepsilon).
        \end{cases}
    \end{equation}
    It is a simple exercise to verify that $\phi_{\opt} = 1 - \mu \gamma$ if $\gamma \leq \gamma_{\crit}$ and $\phi_{\opt} = \gamma L - 1$ if $\gamma \geq \gamma_{\crit}$, which leads us to compute the value of $b$ through a case distinction:
    \begin{itemize}
        \item if $\gamma \leq \gamma_{\crit}$ then $1- \phi_{\opt}^2 = \gamma \mu (2 - \gamma \mu)$ and
        \begin{eqnarray*}
            b &=&
            \gamma(L + \mu) -  \mu L \gamma^2 -  \gamma \mu (2 - \gamma \mu) + \varepsilon  \\
            &=&
            \gamma\left( L + \mu - \mu L \gamma -2 \mu + \gamma \mu^2 \right) + \varepsilon \\
            &=&
            \gamma (1 - \gamma \mu) (L  -  \mu) + \varepsilon \\
            &=&
            \gamma \phi_{\opt} (L  -  \mu) + \varepsilon.
        \end{eqnarray*}
        \item if $\gamma \geq \gamma_{\crit}$ then $1- \phi_{\opt}^2 = \gamma L (2 - \gamma L)$ and
        \begin{eqnarray*}
            b &=&
            \gamma(L + \mu) -  \mu L \gamma^2 -  \gamma L (2 - \gamma L) + \varepsilon  \\
            &=&
            \gamma\left( L + \mu - \mu L \gamma -2 L + \gamma L^2 \right) + \varepsilon \\
            &=&
            \gamma (\gamma L -1) (L  -  \mu) + \varepsilon \\
            &=&
            \gamma \phi_{\opt} (L  -  \mu) + \varepsilon.
        \end{eqnarray*}
    \end{itemize}
    In both cases we obtain $b = \gamma \phi_{\opt} (L  -  \mu) + \varepsilon \geq 0$.

    Since $L > \mu$, we have that $a>0$, so the largest feasible $\alpha$ is the largest root of $P$, if existent. The discriminant $\Delta$ of $P$ is given by
    \begin{eqnarray*}
        \Delta &=& (2b)^2 - 4 ac \\
        &=&
        4\left( \gamma \phi_{\opt} (L  -  \mu) + \varepsilon \right)^2
        - 4 (L-\mu)^2 \gamma^2(\phi^2_{\opt} + \varepsilon) \\
        &=&
        4\gamma^2 \phi_{\opt}^2 (L  -  \mu)^2 + 4\varepsilon^2 
        + 8 \varepsilon \gamma \phi_{\opt} (L  -  \mu) 
        - 4 (L-\mu)^2 \gamma^2 \phi^2_{\opt} 
        - 4 (L-\mu)^2 \gamma^2 \varepsilon \\
        &=&
        4 \left[ \varepsilon^2 +  \varepsilon \left( 2  \gamma \phi_{\opt} (L  -  \mu) - (L-\mu)^2 \gamma^2\right) \right] \\
        &=&
        4 \left[ \varepsilon^2 +  \varepsilon \gamma (L  -  \mu) \left( 2   \phi_{\opt}  - (L-\mu) \gamma \right) \right].
    \end{eqnarray*}
    A simple calculation shows that
    \begin{equation*}
        2   \phi_{\opt}  - (L-\mu) \gamma
        =
        \begin{cases}
            2 - \gamma (L + \mu) & \text{ if } \gamma \leq \gamma_{\crit} \\
            -2 + \gamma (L+\mu) & \text{ if } \gamma \geq \gamma_{\crit}
        \end{cases}
        =
        \vert 2 - \gamma (L + \mu) \vert = (L+\mu) \vert \gamma - \gamma_{\crit} \vert.
    \end{equation*}
    If we introduce $\delta \coloneq \vert \gamma - \gamma_{\crit} \vert$ then we have the simpler and nonnegative expression
    \begin{equation*}
        \Delta = 
         4 \left[ \varepsilon^2 +  \varepsilon \gamma (L  -  \mu) (L + \mu) \delta \right]\ge 0.
    \end{equation*}
    As such, $P$ has real roots. We thus define $\alpha$ to be the largest of the two roots
    \begin{eqnarray*}
        \alpha &=& \frac{2b + \sqrt{\Delta}}{2a} \\
        &=&
        \frac{\gamma \phi_{\opt} (L-\mu) + \varepsilon + \sqrt{\varepsilon^2 +  \varepsilon \gamma (L  -  \mu) (L + \mu) \delta }}{(L-\mu)^2} \\
        &=&
        \omega + \omega_\varepsilon, 
    \end{eqnarray*}
    where we have defined
    \begin{equation*}
        \omega = \frac{\gamma \phi_{\opt}}{L-\mu} \quad \text{ and } \quad 
        \omega_\varepsilon =
        \frac{\varepsilon + \sqrt{\varepsilon^2 +  \varepsilon \gamma (L  -  \mu) (L + \mu) \delta }}{(L-\mu)^2}.
    \end{equation*}

    To conclude the proof, we are now going to verify that this choice of $\alpha$ is feasible, in the sense that it satisfies all the conditions from Theorem \ref{T:lyapunov sufficient conditions strongly convex}.
    The nonnegativity Condition \ref{T:lyapunov sufficient conditions strongly convex:positivity} is satisfied by construction.
    Condition \ref{T:lyapunov sufficient conditions strongly convex:rho} is satisfied as we set $\rho = 0$ and $\alpha_t = \beta_t$.
    Condition \ref{T:lyapunov sufficient conditions strongly convex:monotone at} is equivalent to
    \begin{eqnarray*}
        && 1 \leq 2 \alpha \mu L + \frac{a_{t}}{a_{t+1}} 
        = 
         2 \alpha\mu L + \phi^2.
    \end{eqnarray*}
    We have $\alpha = \omega + \omega_\varepsilon \geq \omega$ and $\phi^2 \geq \phi^2_{\opt}$, and the above is thus true if
    \begin{eqnarray*}
        &  & 1 \leq 
        \mu L 2 \omega + \phi^2_{\opt}  \\
        &\iff &
        1 \leq 
        \frac{2 \mu L \gamma \phi_{\opt}}{L - \mu} +\phi^2_{\opt} \\
        & \iff &
        0 \leq 2 \mu L \gamma \phi_{\opt} - (1 - \phi^2_{\opt})(L-\mu).
    \end{eqnarray*}
    To assess whether this inequality is true, we are going to consider two cases:
    \begin{itemize}
        \item if $\gamma \leq \gamma_{\crit}$, then $\phi_{\opt} = 1 - \gamma \mu$, so the inequality becomes
        \begin{eqnarray*}
            0 & \leq &
            2 \mu L \gamma (1 - \gamma \mu) - \gamma \mu (2 - \gamma \mu)(L-\mu) \\
            &=&
            \mu \gamma (- \mu \gamma L - \gamma \mu^2 + 2 \mu) \\
            &=&
            \mu^2 \gamma (2 - \gamma(L + \mu)) \\
            &=&
            \mu^2 \gamma (\mu + L) \delta.
        \end{eqnarray*}
        \item  if $\gamma \geq \gamma_{\crit}$, then $\phi_{\opt} = \gamma L - 1$, so the inequality becomes
        \begin{eqnarray*}
            0 & \leq &
            2 \mu L \gamma (\gamma L - 1) - \gamma L (2 - \gamma L)(L-\mu) \\
            &=&
            L \gamma (\mu \gamma L + \gamma L^2 - 2 L) \\
            &=&
            -L^2 \gamma (2 - \gamma(L + \mu)) \\
            &=&
            L^2 \gamma (\mu + L) \delta.
        \end{eqnarray*}
    \end{itemize}
    In both cases the inequality is verified, therefore Condition \ref{T:lyapunov sufficient conditions strongly convex:monotone at} holds true.
    Let us now turn to 
    Condition \ref{T:lyapunov sufficient conditions strongly convex:annoying}, which, in our context, is equivalent to
    \begin{equation*}
        0 \leq 2 \alpha - \gamma^2,
    \end{equation*}
    which holds true as
    \begin{eqnarray*}
        2 \alpha - \gamma^2
        &=&
        2 \omega_\varepsilon + 2 \omega - \gamma^2
        =
        2 \omega_\varepsilon 
        + \frac{\gamma}{L-\mu} 
        \left( 2 \phi_{\opt} - \gamma(L-\mu) \right) \\
        &=&
        2 \omega_\varepsilon 
        + \frac{\gamma}{L-\mu} 
        (L+ \mu) \delta.
    \end{eqnarray*}
    Condition \ref{T:lyapunov sufficient conditions strongly convex:curious} is true by the choice of $\alpha$.
    Finally, Condition \ref{T:lyapunov sufficient conditions strongly convex:variance} holds true provided $2\alpha-\gamma^2>0$,  which is verified through the above computation provided $\varepsilon>0$ or $\delta\neq 0$.
    In this case,
    \begin{eqnarray*}
        e&=&
        \gamma^2 \frac{2 \omega + 2 \omega_\varepsilon}{2 \omega_\varepsilon 
        + \frac{L+ \mu}{L-\mu} 
         \gamma \delta} \\
        &=&
        \gamma^2 \frac{2 \gamma \phi_{\opt} + 2 (L-\mu)\omega_\varepsilon}{2 \omega_\varepsilon (L- \mu) 
        + ({L+ \mu})
         \gamma \delta} \\
        &=&
        \gamma^2 \frac{\gamma ((L+\mu) \delta + \gamma(L-\mu)) + 2 (L-\mu)\omega_\varepsilon}{2 \omega_\varepsilon (L- \mu) 
        + ({L+ \mu})
         \gamma \delta} \\
         &=&
         \gamma^2 \left( 1 + \frac{\gamma^2(L- \mu)}{2 \omega_\varepsilon (L- \mu) 
        + ({L+ \mu})
         \gamma \delta} \right).
    \end{eqnarray*}
    The bound on $e_T^{sum}$ is a direct consequence of the fact that $a_t$ is a geometric sequence.
    The upper bound for $e$ comes from the lower bound $\omega_\varepsilon \geq \frac{2 \varepsilon}{(L- \mu)^2}$.
\end{proof}

We note the above proposition only holds true for $\mu < L$, thus not covering the case $L=\mu$. We cover this case in a separate proposition.

\begin{proposition}[Lyapunov parameters. Strongly convex case, any step-size, $L = \mu$]\label{P:lyapunov strongly convex any step-size L=mu}
    Let $ \gamma L \in (0,2)$, with $L = \mu$.
    Let $\varepsilon > 0$.
    Consider the parameters $\rho, a_t, e_t, \alpha_t, \beta_t$ defined by
    \begin{itemize}
        \item $\rho =0$,
        \item $a_t = \phi^{2(T-t)}$, where $\phi^2 = \phi_{\opt}^2 + \varepsilon$, with 
        \begin{equation*}
            \phi_{\opt} = \max\{ 1- \gamma \mu ; \gamma L -1\}
            \quad \text{ and } \quad 
            \varepsilon \in (0,1-\phi^2_{\opt}),
        \end{equation*}
        \item $\alpha_t = \beta_t$ and $\alpha_t = \alpha a_{t+1}$ where 
    \begin{equation*}
        \alpha > \max \left\{ \frac{\gamma^2}{2} \left(1+ \frac{\phi_{\opt}^2}{\varepsilon} \right) ,  \frac{1 - \phi^2_{\opt} - \varepsilon}{2 L^2}, \frac{1}{\varepsilon} \right\},
    \end{equation*}
        \item $e_t = e a_{t+1}$ where
        \begin{equation*}
            e = \frac{2 \gamma^2 \alpha}{2 \alpha - \gamma^2}.
        \end{equation*}
    \end{itemize}
    Then the sufficient Lyapunov conditions of Theorem \ref{T:lyapunov sufficient conditions strongly convex} are verified.
    Moreover, we have
    \begin{equation*}
        \phi \xrightarrow{\varepsilon \to 0} \phi_{\opt}, \quad
        \alpha \xrightarrow{\varepsilon \to 0} +\infty, \quad
        e \xrightarrow{\varepsilon \to 0} \gamma^2.
    \end{equation*}
\end{proposition}

\begin{proof}
    The arguments in this proof are mostly recycled from Proposition \ref{P:lyapunov strongly convex any step-size}.
    In particular, we shall use the same values for $\rho$ and $a_t$.
    Regarding the value of $\alpha$, we impose $a \alpha^2 -2b \alpha +c \leq 0$ where $a,b,c$ are defined in Equation \eqref{D:sgd strong convex curious polynomial coefficients}, which we recall for convenience:
    \begin{equation*}
    \begin{cases}
        a &= (L-\mu)^2, \\
        b &= \gamma(L + \mu) -  \mu L \gamma^2 - (1- \phi^2), \\
        c &= \gamma^2 ( \phi_{\opt}^2 + \varepsilon).
    \end{cases}
    \end{equation*}
    Because $L= \mu$, it holds that $a=0$, and that
    \begin{equation*}
        b = 2\gamma L -  L^2 \gamma^2 - (1 - \phi_{\opt}^2) + \varepsilon
        =2\gamma L -  L^2 \gamma^2 - \gamma L (2 - \gamma L) + \varepsilon
        = \varepsilon.
    \end{equation*}
    Therefore, $\alpha$ must satisfy $c \leq 2b \alpha$, that is
    \begin{equation}\label{eq:sgd strong convex alpha lower bound}
        \alpha \geq \frac{c}{2b}
        =
        \frac{\gamma^2 ( \phi_{\opt}^2 + \varepsilon)}{2 \varepsilon}
        =
        \frac{\gamma^2}{2} \left(1+ \frac{\phi_{\opt}^2}{\varepsilon} \right).
    \end{equation}
    It then remains to verify the conditions of Theorem \ref{T:lyapunov sufficient conditions strongly convex}.
    Conditions \ref{T:lyapunov sufficient conditions strongly convex:positivity} and \ref{T:lyapunov sufficient conditions strongly convex:rho} are readily satisfied.
    Condition \ref{T:lyapunov sufficient conditions strongly convex:monotone at} can be rewritten as 
    \begin{equation*}
        \alpha \geq \frac{1 - \phi^2_{\opt} - \varepsilon}{2 L^2},
    \end{equation*}
    thus justifying our definition of $\alpha$.
    Condition \ref{T:lyapunov sufficient conditions strongly convex:annoying} requires that $\alpha \geq \tfrac{\gamma^2}{2}$, which is already true because of Equation \eqref{eq:sgd strong convex alpha lower bound}.
    Condition \ref{T:lyapunov sufficient conditions strongly convex:curious} is satisfied under Equation \eqref{eq:sgd strong convex alpha lower bound}.
    Finally, Condition \ref{T:lyapunov sufficient conditions strongly convex:variance} is satisfied with 
        \begin{equation*}
            e = \frac{2 \gamma^2 \alpha}{2 \alpha - \gamma^2},
        \end{equation*}
        where the denominator is strictly positive because Equation \eqref{eq:sgd strong convex alpha lower bound} holds strictly.
\end{proof}

\begin{theorem}[Strongly convex case, general step-size]\label{T:SGD strong convex general step-size epsilon}
Let Assumptions \ref{Ass:convex smooth} and \ref{Ass:bounded solution variance} hold, with $\mu > 0$.
Let $\phi_\opt = \max\{ 1- \gamma \mu ; \gamma L -1\}$, $0 \leq \varepsilon < 1 - \phi_\opt^2$ and $\gamma_{\crit} = \tfrac{2}{\mu + L}$.
Let $(x_t)$ be generated by \eqref{S1::eq:SGD}, with $\gamma L \in (0, 2)$. 
Assume that either $\varepsilon >0$ or $\gamma \neq \gamma_{\crit}$.
Then, for every $T \geq 1$,
\begin{equation*}
    \Exp{\| x_T-x_*\|^2}\le \phi^{2T}\cdot \|x_0-x_*\|^2 + \frac{1 - \phi^{2T}}{1- \phi^2} e \sigma_*^2,
\end{equation*}
where $\phi^2 = \varepsilon + (\max\{ 1- \gamma \mu ; \gamma L -1\})^2 \in [0,1)$, and $e=\gamma^2$ if $L=\mu$, while if $L>\mu$, $e$ is defined in Proposition \ref{P:lyapunov strongly convex any step-size} and verifies 
    \begin{equation*}
        e \leq \gamma^2 \left( 1 + \frac{\gamma^2 (L - \mu)^2}{4 \varepsilon + (L+ \mu)(L- \mu) \gamma \vert \gamma - \gamma_{\crit} \vert} \right)
        =
        \gamma^2 \frac{\gamma^2_{\crit}\varepsilon + \gamma \gamma_{\crit} \phi_\opt \phi_{\crit}}{\gamma^2_{\crit}\varepsilon + \gamma \phi_{\crit} \vert \gamma - \gamma_{\crit} \vert}.
    \end{equation*}
    In particular, the bounds of Theorems \ref{S3::thm:strong} and \ref{S3::thm:strong_opt_bias} hold true.
\end{theorem}

\begin{proof}
We separate the proof into two cases, namely the case $L>\mu$ and the case $L=\mu$.
\begin{itemize}
    \item Consider $L>\mu$. Proposition \ref{P:lyapunov strongly convex any step-size}, combined with Theorem \ref{T:lyapunov sufficient conditions strongly convex}, provides parameters ensuring that the Lyapunov energy decreases with, in particular, $a_T = 1$, $a_0 = \phi^{2T}$.
    We deduce the wanted bound from Lemma \ref{L:SGD bound from lyapunov decrease strongly convex}.
    \item Consider $L=\mu$. Let us consider first $\varepsilon > 0$. Combining Lemma \ref{L:SGD bound from lyapunov decrease strongly convex}, Theorem \ref{T:lyapunov sufficient conditions strongly convex} and Proposition \ref{P:lyapunov strongly convex any step-size L=mu} we obtain
    \begin{equation*}
        \Exp{\| x_T-x_*\|^2}\le \phi^{2T}\cdot \|x_0-x_*\|^2 + \frac{1 - \phi^{2T}}{1- \phi^2} e_\varepsilon \sigma_*^2,
    \end{equation*}
    where $e_\varepsilon$ converges to $\gamma^2$ and $\phi\to \phi_{\opt}$ when $\varepsilon \to 0$, thus yielding the wanted bound when $\varepsilon=0$.
    If we wish to prove the claim for $\varepsilon>0$, it is enough to realize that $\phi_{\opt} \leq \phi$, together with the fact that $\tfrac{1 - \phi^{2T}}{1- \phi^2}$ is nondecreasing with $\phi$, and the fact that we can take $\alpha\to +\infty$, in which case $e_\varepsilon$ converges to $\gamma^2$ for any $\varepsilon>0$.
\end{itemize}
The bounds of Theorems \ref{S3::thm:strong} and \ref{S3::thm:strong_opt_bias} now follow by either setting $\varepsilon=0$ or using that $\gamma=\gamma_{\crit}$. 
\end{proof}

\begin{theorem}[Strongly convex case, critical step-size, $L = \mu$]\label{T:SGD strong convex critical step-size mu equals L}
Let Assumptions \ref{Ass:convex smooth} and \ref{Ass:bounded solution variance} hold, with $L = \mu > 0$.
Let $(x_t)$ be generated by \eqref{S1::eq:SGD}, with $\gamma = \gamma_{\crit} = \tfrac{2}{L+\mu}$. 
Then, for every $T \geq 1$,
    \begin{equation*}
        \Exp{\| x_T-x_*\|^2}\le \gamma^2 \sigma_*^2.
    \end{equation*}
\end{theorem}

\begin{proof}
    Combining Lemma \ref{L:SGD bound from lyapunov decrease strongly convex}, Theorem \ref{T:lyapunov sufficient conditions strongly convex} and Proposition \ref{P:lyapunov strongly convex any step-size L=mu}, we have, for any $\varepsilon > 0$
    \begin{equation*}
        \Exp{\| x_T-x_*\|^2}\le \phi_\varepsilon^{2T}\cdot \|x_0-x_*\|^2 + \frac{1 - \phi_\varepsilon^{2T}}{1- \phi_\varepsilon^2} e_\varepsilon \sigma_*^2,
    \end{equation*}
    where $e_\varepsilon \to \gamma^2$ as $\varepsilon \to 0$, and where $\phi_\varepsilon^2 = \max\{1-\gamma \mu, \gamma L - 1\}^2 + \varepsilon$. 
    Note that $\gamma = \tfrac{2}{L +\mu}$ implies that $\max\{1-\gamma \mu, \gamma L - 1\} = 0$ and so $\phi_\varepsilon^2 = \varepsilon$.
    Hence, taking the limit as $\varepsilon \to 0$ gives the desired bound.
\end{proof}

Note that in the case \(\mu=L\), every \(f_i\) is a quadratic function of the form \(f_i = \frac{\mu}{2}\|\cdot-x_i^*\|^2+c_i.\)
Consequently, the SGD iterates satisfy the following exact recursion:
\[
\E[\|x_{t+1}-x_*\|^2]
=(1-\gamma\mu)^2\E[\|x_t-x_*\|^2]
+\gamma^2\sigma_*^2.
\]
Hence, our Lyapunov analysis recovers the bound obtained by unfolding the above equality.

We now show the bias-optimality claim of Theorem \ref{S3::thm:strong}.

\begin{proposition}[Bias-optimality. Strongly convex case]
    The bias obtained in the bound of Theorem \ref{S3::thm:strong} is optimal, in the sense that there exists a strongly convex and smooth problem for which $\mathcal I=\{1\}$ and the inequality is satisfied with equality.
\end{proposition}
\begin{proof}
    It suffices to consider a quadratic function given by $f(x) = \frac{\mu}{2}\|x\|^2$ if $\gamma \le \frac{2}{L+\mu}$, and by $f(x)=\frac{L}{2}\|x\|^2$ otherwise, and any $x_0 \in \H$. The iterates $(x_t)$ of gradient descent are given by $x_t = (1-\gamma \mu)^t x_0$ or $x_t = (1-\gamma L)^t x_0$.
    Noting that $x_* = 0$ in either case, we obtain that
    \[
        \|x_T-x_*\|^2 = \begin{cases}
        (1-\gamma \mu)^{2T} \cdot \|x_0 - x_*\|^2 &\quad \text{if $\gamma \le \frac{2}{L+\mu}$} \\
        (1-\gamma L)^{2T} \cdot \|x_0 - x_*\|^2 &\quad \text{if $\gamma \ge \frac{2}{L+\mu}$}
    \end{cases}=\max(1-\gamma \mu, \gamma L - 1)^{2T}\cdot \|x_0 - x_*\|^2,
    \]
    matching the bias term obtained in Theorem \ref{S3::thm:strong}.
\end{proof}

We end by showing that, at the critical step-size, the optimal bias cannot be achieved whilst retaining a finite variance coefficient. 

\begin{proof}[of Proposition \ref{thm:sing_str}]
    Given $b>0$, consider $\H=\R$ and two functions 
    \[
        f_1(x)=\frac{\mu}{2}x^2-b x\quad \text{and}\quad f_2(x)=\frac{L}{2}x^2+bx,
    \]
    which are each sampled with probability $1/2$. Both functions are $\mu$-strongly convex and $L$-smooth, and
    \[
        f(x)=\frac{L+\mu}{4}x^2, \quad x_*=0, \quad \text{and}\quad \sigma_*^2=b^2.
    \]
    Now note that
    \[
        x_{t+1}=\begin{cases}
            x_t-\frac{2}{L+\mu}(\mu x_t-b)\quad &\text{with probability $1/2$} \\
            x_t-\frac{2}{L+\mu}(L x_t+b)\quad &\text{with probability $1/2$}
        \end{cases}\quad  = \varepsilon_t (\phi_{\crit}x_t+\gamma_{\crit}b),
    \]
    where $\varepsilon_t=+1$ with probability $1/2$ and $\varepsilon_t=-1$ with probability $1/2$. Specifically, $\ExpD{t}{x_{t+1}}=0$ for all $t\ge 0$, and
    \[
        \ExpD{t}{x_{t+1}^2}=\ExpD{t}{(\phi_{\crit}x_t+\gamma_{\crit}b)^2}=\phi_{\crit}^2x_t^2+2\phi_{\crit}\gamma_{\crit}bx_t+\gamma_{\crit}^2b^2.
    \]
    Taking total expectation yields, for $t\ge 1$,
    \[
        \Exp{x_{t+1}^2}=\phi_{\crit}^2\Exp{x_t^2}+\gamma_{\crit}^2b^2\quad \text{and}\quad \Exp{x_1^2}=\phi_{\crit}^2x_0^2+2\phi_{\crit}\gamma_{\crit}bx_0+\gamma_{\crit}^2b^2,
    \]
    which yields
    \begin{align*}
        \Exp{\|x_T-x_*\|^2}=\Exp{x_T^2}
        &=\phi_{\crit}^{2T-2}\Exp{x_1^2}+\gamma_{\crit}^2b^2\sum_{t=0}^{T-2}\phi_{\crit}^{2t}\\
        &=\phi_{\crit}^{2T}x_0^2+2\phi_{\crit}^{2T-1}\gamma_{\crit}bx_0+\phi_{\crit}^{2T-2}\gamma_{\crit}^2b^2+\gamma_{\crit}^2b^2\sum_{t=0}^{T-2}\phi_{\crit}^{2t} \\
        &=\phi_{\crit}^{2T}x_0^2+2\phi_{\crit}^{2T-1}\gamma_{\crit}bx_0+\gamma_{\crit}^2b^2\sum_{t=0}^{T-1}\phi_{\crit}^{2t}.
    \end{align*}
    As such, if a bound of the form \eqref{eq:str_to_match} were to exist, it would hold for this problem, and therefore it would need to hold that
    \[
        \frac{2\phi_{\crit}^{2T-1}\gamma_{\crit}x_0}{b}+\gamma_{\crit}^2\sum_{t=0}^{T-1}\phi_{\crit}^{2t}\le V,
    \]
    for all $b>0$. Fixing $x_0>0$ and letting $b\to 0$, the left-hand side diverges, contradicting the finiteness of $V$.
\end{proof}


\section{SGD with Mini-Batching}\label{SA::sec:mini}\label{S:minibatch}\label{SA::sec:extensions}

In practice, stochastic gradient descent (SGD) is often used to solve finite-sum optimization problems, where mini-batching is a common strategy to improve its practical performance. 
Mini-batching is not explicitly considered in our presentation, but it can also be handled very easily.
Indeed, it suffices to view the mini-batched algorithm as vanilla SGD applied to an equivalent problem and to compute the corresponding problem constants, namely $L$, $\mu$, and $\sigma_*^2$. 
This black-box approach of extending SGD convergence results to different algorithmic variants is also discussed in \cite{gower_sgd_2019}. For the sake of completeness, we briefly explain how to handle mini-batching below.

Consider a finite family of functions $f_1, \dots, f_n$ which we assume to be $\mu_i$-strongly convex and $L_i$-smooth, for $\mu_i \geq 0$ and $L_i >0$. 
We consider the problem of minimizing
\[
f(x) = \frac{1}{n}\sum_{i=1}^n f_i(x),
\]
that is, minimizing the expectation of functions $f_i$ according to the uniform distribution.
We want to study the mini-batch SGD algorithm, which computes at each iteration a mini-batch of gradients:
\begin{equation}\label{D:SGD miibatch}\tag{SGD$_b$}
    x_{t+1} = x_t - \frac{\gamma}{b} \sum_{i \in B_t} \nabla f_{i}(x_t).
\end{equation}
Here the batch $B_t$ is sampled i.i.d. and uniformly among the subsets of size $b$ of $\{1, \dots, n\}$.
It is a simple exercise to see that this algorithm is exactly \eqref{S1::eq:SGD} applied to the problem 
\begin{equation*}
    \min_x \ f(x)  =\ExpD{\mathcal{U}_b}{\hat f_B(x)},
    \text{ where }
    \hat f_B(x) \coloneqq \frac{1}{b} \sum_{i \in B} f_{i}(x),
\end{equation*}
and where the expectation is taken with respect to $\mathcal{U}_b$ which is the uniform law over
\begin{equation*}
    \mathcal{B} \coloneqq  \{ B \subset \{1, \dots, n \} \ : \ \vert B \vert = b \}.
\end{equation*}
We note that Assumption \ref{Ass:strongly convex smooth} holds for $\mu=\min_{i}\mu_i$ and $L=\max_iL_i$. Moreover, one can compute that, with the convention that $\binom{n-2}{b-2} = 0$ for $b=1$, 
\begin{align*}
    \sigma_*^2 & =
    \ExpD{\mathcal{U}_b}{\Vert \nabla \hat f_B(x_*) \Vert^2}
    = \frac{1}{|\mathcal{B}|}\sum_{B\in \mathcal B} \left\|\frac1b\sum_{i\in B}\nabla f_i(x_*)\right\|^2
    = \frac{1}{\binom{n}{b}\cdot b^2}\sum_{B\in \mathcal B}\sum_{i, j\in B}\langle \nabla f_i(x_*), \nabla f_j(x_*)\rangle \\
    & = \frac{1}{\binom{n}{b}\cdot b^2}\left[\binom{n-2}{b-2}\sum_{i, j=1}^n \langle \nabla f_i(x_*), \nabla f_j(x_*)\rangle + \left[\binom{n-1}{b-1}-\binom{n-2}{b-2}\right]\sum_{i=1}^n\|\nabla f_i(x_*)\|^2\right]
    = \frac{n-b}{nb(n-1)} \sum_{i=1}^n \Vert \nabla f_i(x_*) \Vert^2.
\end{align*}
\section{Auxiliary Results} \label{SD::sec}

\begin{proposition} \label{S1::prop:finite expectations}
    Consider Problem \ref{Pb:main}. Let Assumptions \ref{Ass:strongly convex smooth} and \ref{Ass:bounded solution variance} hold true, and let $X \in \mathcal{H}$ be independent of $i$ such that $\E[\|X\|^2] < +\infty$. Then the following expectations are finite
    \begin{equation} \label{SD::prop:eq}
        \E[\|\nabla f_i(X)\|^2], \quad \E[\langle \nabla f_i(X), \nabla f_i(x_*)\rangle], \quad \E[\langle \nabla f_i(X), X - x_*\rangle], \quad \E[f(X) - \inf f].
    \end{equation}
    Consequently, if \eqref{S1::eq:SGD} is initialized at a deterministic $x_0 \in \mathcal{H}$, then $\E[\|x_t\|^2]<+\infty$ for every $t\in\N$.
\end{proposition}

\begin{proof}
    First, since \(\mathcal H\) is separable, the measurability of \(i\mapsto \nabla f_i(x)\) for every \(x\in\mathcal H\), together with the continuity of \(x\mapsto \nabla f_i(x)\) for every \(i\in\mathcal I\), ensures by \cite[Theorem 2.5.21]{papageorgiou2019nonlinear} that the mapping \((i,x)\mapsto \nabla f_i(x)\) is jointly measurable. Consequently, the expectations appearing in \eqref{SD::prop:eq} are well defined. Further, from $L-$smoothness, 
    \[
    \E[\|\nabla f_i(X)\|^2] \leq 2\E[\|\nabla f_i(X) - \nabla f_i(x_*)\|^2] + 2\sigma^2_* \leq 2L^2 \E[\|X - x_*\|^2] + 2\sigma^2_* < +\infty.
    \]
    The finiteness of the two inner products follows from the latter and the Cauchy-Schwarz inequality. Finally, we may use the descent lemma to write
    \[
    \E[f(X) - \inf f] \leq \frac{L}{2}\E[\|X-x_*\|^2] < +\infty.
    \]
    The result for the SGD iterates follows by induction, since 
    \[
    \E[\|x_{t+1}\|^2] = \E[\|x_t - \gamma \nabla f_{i_t}(x_t)\|^2] \leq 2\E[\|x_t\|^2] + 2\gamma^2\E[\|\nabla f_{i_t}(x_t)\|^2] < +\infty.
    \]
\end{proof}

\begin{proposition} \label{SD::prop:inf measurable}
    Consider Problem \ref{Pb:main} and  let assumptions \ref{Ass:strongly convex smooth} hold. Then, the mapping
    \(
    i \to \inf f_i,
    \)
    is measurable. 
\end{proposition}

\begin{proof}
    Since each $f_i$ is continuous and $\mathcal{H}$ is separable, $\inf_{x \in \mathcal H} f_i(x) = \inf_{x \in D} f_i(x)$, where $D$ is a dense countable subset of $\mathcal H$. Hence, for all $a \in \R$,
    \[
    \{i \in \mathcal I \ | \ \inf f_i < a\} = \bigcup_{x \in D} \{i \in \mathcal I \ | \  f_i(x) < a\},
    \]
    which is measurable, as a countable union of measurable sets.
\end{proof}

\bibliographystyle{spmpsci}
\bibliography{references}

\end{document}